\documentclass[11pt]{amsart}

\usepackage{epigamath}

\usepackage[english]{babel}

\numberwithin{equation}{section}

\usepackage[shortlabels]{enumitem}
\usepackage{amsmath,amsfonts,amssymb,amsthm}

\usepackage{caption}
\usepackage{cancel}
\usepackage{subcaption}
\DeclareCaptionSubType*[arabic]{figure}  
\usepackage{chngcntr}
\counterwithout{subfigure}{section}

\usepackage[new]{old-arrows}

\usepackage{graphicx}
\usepackage{color}
\usepackage[all]{xy}
\usepackage{url}
\usepackage{setspace}
\usepackage{booktabs}
\usepackage{longtable}
\usepackage{latexsym}
\usepackage{comment}

\usepackage{longtable}
\usepackage{framed}

\usepackage{tikz}
\usetikzlibrary{cd}
\usetikzlibrary{patterns}
\usetikzlibrary{shapes}
\usetikzlibrary{decorations.fractals}

\theoremstyle{plain}
\newtheorem*{introthm}{Theorem}
\newtheorem{theorem}{Theorem}[section]
\newtheorem{lemma}[theorem]{Lemma}
\newtheorem{proposition}[theorem]{Proposition}
\newtheorem{corollary}[theorem]{Corollary}

\theoremstyle{definition}
\newtheorem{definition}[theorem]{Definition}

\theoremstyle{remark}
\newtheorem{example}[theorem]{Example}
\newtheorem{remark}[theorem]{Remark}

\DeclareMathOperator{\spec}{Spec}

\DeclareMathOperator{\ord}{ord}
\DeclareMathOperator{\supp}{supp}

\DeclareMathOperator{\TV}{\mathbb{T}\mathbb{V}}
\DeclareMathOperator{\cone}{cone}

\newcommand{\lra}{\longrightarrow}

\DeclareMathOperator{\length}{length}

\newcommand{\PP}{\mathbb P}
\newcommand{\A}{\mathbb A}

\newcommand{\N}{\mathbb N}
\newcommand{\R}{\mathbb R}
\newcommand{\T}{\mathbb T}
\newcommand{\V}{\mathbb V}
\newcommand{\Z}{\mathbb Z}
\newcommand{\C}{\mathbb C}
\newcommand{\Q}{\mathbb Q}

\newcommand{\CL}{{\mathcal L}}

\definecolor{intOrange}{rgb}{1.0,.310,.0} 
\newcommand{\til}[1]{\widetilde{#1}}

\newcommand{\spann}{\operatorname{span}}

\newcommand{\kst}{\,|\;}

\newcommand{\surj}{\rightarrow\hspace{-0.8em}\rightarrow}

\newcommand{\bM}{\ko{M}}

\definecolor{skin}{HTML}{FFECC9}
\definecolor{pumpkin}{HTML}{FEDFA9}
\definecolor{piggy}{HTML}{FFB99D}
\definecolor{fiolet}{HTML}{CD8F9C}
\definecolor{granat}{HTML}{677081}
\definecolor{ciemnyblekit}{HTML}{91A1B8}

\definecolor{oliwkowy}{HTML}{627037}
\definecolor{ciemnazielen}{HTML}{394D2E}
\definecolor{ciemnyfiolet}{HTML}{424444}
\definecolor{mocnyfiolet}{HTML}{717299}
\definecolor{jasnyfiolet}{HTML}{B0ABCC}
\definecolor{bladyfiolet}{HTML}{C9C7DB}

\definecolor{lightblue}{RGB}{135,206,250}

\definecolor{cNewton}{RGB}{210,75,90}
\definecolor{cNef}{RGB}{50,135,90} 
\definecolor{cdiv}{RGB}{10,100,155}
\definecolor{cthet}{RGB}{225,220,30}
\definecolor{ckothet}{RGB}{225,130,10}
\definecolor{cnob}{RGB}{165,115,170}

\newcommand{\curv}{\textcolor{black}{C}} 
\newcommand{\curvbar}{\textcolor{black}{C'}}
\newcommand{\latn}{N} 
\newcommand{\latm}{M} 
\newcommand{\var}{{X}} 
 
\newcommand{\op}{\textcolor{black}{m}}

\newcommand{\poly}{\textcolor{black}{\Delta}} 
\newcommand{\ds}{\textcolor{black}{D}}

\newcommand{\divpol}[1]{\textcolor{black}{\poly({#1})}}
\newcommand{\newton}{\operatorname{\textcolor{black}{newt}}}

\newcommand{\newpolv}{\textcolor{black}{\poly^{\newton}}} 

\newcommand{\nef}{\operatorname{nef}}

\newcommand{\nefpolv}{\textcolor{black}{\poly^{\nef}}} 

\newcommand{\nob}[2]{{\poly_{#1}(#2)}}
\newcommand{\direcm}{\textcolor{black}{v_\latm}} 
\newcommand{\direcn}{\textcolor{black}{v_\latn}} 
\newcommand{\direcnbot}{\textcolor{black}{v^{\bot}_\latn}} 
\newcommand{\flag}{Y_{\bullet}} 
  
\newcommand{\ray}{\rho}

\newcommand{\wid}[1]{\operatorname{width}_{#1}} 

\newcommand{\binomf}{\textcolor{black}{f}} 
\DeclareMathOperator{\vplus}{r_{{\max}}}
\DeclareMathOperator{\vminus}{r_{{\min}}} 
\DeclareMathOperator{\coneplus}{\sigma_{{\max}}}
\DeclareMathOperator{\coneminus}{\sigma_{{\min}}} 
\newcommand{\nablaplus}{\Delta_{{\max}}}
\newcommand{\nablaminus}{\Delta_{{\min}}} 
\newcommand{\quadplus}{\square_{{\max}}}
\newcommand{\quadminus}{\square_{{\min}}}

\newcommand{\thetcut}{\divpol{\ds}^{C}}
\newcommand{\thetcutelka}{\divpol{\ds}^{C(\el,\ka)}}
\newcommand{\dline}{\divpol{\ds}^{\direcn}}
\newcommand{\dlineone}{\textcolor{black}{v_1}}
\newcommand{\dlinetwo}{\textcolor{black}{v_2}}
\newcommand{\edge}{e}

\newcommand{\ebigone}{\textcolor{black}{e_1^{+}}}
\newcommand{\ebigtwo}{\textcolor{black}{e_2^{+}}}
\newcommand{\esmallone}{\textcolor{black}{e_1^{-}}}
\newcommand{\esmalltwo}{\textcolor{black}{e_2^{-}}}

\newcommand{\conbig}{\textcolor{black}{\con^+}}
\newcommand{\consmall}{\textcolor{black}{\con^-}}

\newcommand{\el}{\ell}
\newcommand{\ka}{k}
\newcommand{\de}{d}
\newcommand{\ee}{e}
\newcommand{\qu}{q}
\newcommand{\quhat}{\hat{\qu}}
\newcommand{\debarelka}{\textcolor{black}{\ko{\de(\el,\ka)}}}
\newcommand{\eebarelka}{\textcolor{black}{\ko{\ee(\el,\ka)}}}
\newcommand{\vau}{L(\el,\ka)}

\DeclareMathOperator{\orb}{orb}
\newcommand{\toric}{\T\V}  
\newcommand{\CO}{{\mathcal O}} 
\newcommand{\ko}{\overline}
 
\newcommand{\kG}{\Gamma}
\newcommand{\kothet}{\textcolor{black}{{\Xi(\el,\ka)}}}

\newcommand{\fan}{\Sigma} 
 
\newcommand{\tor}{\mathbb{T}} 
\newcommand{\val}{\operatorname{val}}

\renewcommand{\div}{\operatorname{div}}

\DeclareMathOperator{\conv}{conv}
\newcommand{\istar}{C}
\DeclareMathOperator{\con}{\sigma} 
\DeclareMathOperator{\vect}{u} 
\DeclareMathOperator{\vectK}{\vect}
  
\newcommand{\interior}{\operatorname{int}} 
\newcommand{\sg}{S}

\newcommand{\hilb}{{\mathcal H}}

\newcommand{\CLlk}{\textcolor{black}{{\mathcal L({\el},{\ka}})}}
\newcommand{\CLlktor}{\textcolor{black}{{{\mathcal L'}({\el},{\ka})}}}
\newcommand{\vautor}{\textcolor{black}{{{L'}({\el},{\ka})}}}

\newcommand{\thet}[1][{\el},{\ka}]{\textcolor{black}{\Theta(#1)}}
\newcommand{\thetq}{\textcolor{black}{\Theta(1,\qu)}}

\newcommand{\colonequals}{:=}
\newcommand{\equalscolon}{=:}

\setlist[enumerate,1]{label={\rm(\roman*)}, ref={\rm\roman*}} 

\EpigaVolumeYear{8}{2024} \EpigaArticleNr{6} \ReceivedOn{May 31, 2023}
\InFinalFormOn{November 23, 2023}
\AcceptedOn{December 15, 2023}

\title{On the finite generation of valuation semigroups on toric surfaces}
\titlemark{Finite generation of toric valuation semigroups}

\author{Klaus Altmann}
\address{Mathematisches Institut, Freie Universit\"at Berlin,
Arnimallee 3, 14195 Berlin, Germany}
\email{altmann@math.fu-berlin.de}
 
\author{Christian Haase}
\address{Mathematisches Institut, Freie Universit\"at Berlin, Arnimallee 3, 14195 Berlin, Germany}
\email{haase@math.fu-berlin.de}

\author{Alex K\"uronya}
\address{Institut f\"ur Mathematik, Goethe-Universit\"at Frankfurt,  Robert-Mayer-Str. 6, 60325 Frankfurt am Main, Germany}
\email{kuronya@math.uni-frankfurt.de}

\author{Karin Schaller}
\address{Mathematisches Institut, Freie Universit\"at Berlin, Arnimallee 3, 14195 Berlin, Germany}
\email{kschaller@math.fu-berlin.de}

\author{Lena Walter}
\address{Mathematisches Institut, Freie Universit\"at Berlin, 
 Arnimallee 3, 14195 Berlin, Germany}
\email{lenawalter@math.fu-berlin.de}

\authormark{K.~Altmann, C.~Haase, A.~K\"uronya, K.~Schaller, L.~Walter}

\AbstractInEnglish{We provide a combinatorial criterion for the finite generation of a valuation semigroup associated with an ample divisor on a smooth toric surface and a non-toric valuation of maximal rank. 
As an application, we construct a lattice polytope such that none of the  valuation semigroups of the associated polarized toric variety coming from one-parameter subgroups and centered at a non-toric point are finitely generated.}


\acknowledgement{The first, second, fourth, and fifth authors have been partially supported by the Polish-German grant ``ATAG -- Algebraic Torus Actions: Geometry and Combinatorics'' [project number 380241778] of the Deutsche Forschungsgemeinschaft (DFG). The third author acknowledges support by the Deutsche For\-schungs\-gemeinschaft (DFG) through the Collaborative Research Centre TRR 326 ``Geometry and Arithmetic of Uniformized Structures'' [project number 444845124].}

\begin{document}


\maketitle

\begin{prelims}

\DisplayAbstractInEnglish

\bigskip

\DisplayKeyWords

\medskip

\DisplayMSCclass

\end{prelims}


\newpage

\setcounter{tocdepth}{1}

\tableofcontents


\section{Introduction}
\label{intro}
 Finite generation of semigroups or rings arising from geometric situations has been a question of interest for a long time. As a salient example, we can recall the finite generation of canonical or adjoint rings from birational geometry, which motivated the field through the minimal model program  for several decades; \textit{cf.}~\cite{BCHM}. The question of finite generation of valuation semigroups arising from Newton--Okounkov theory appears to be equally difficult in general, with little progress beyond the completely toric situation, but potentially great benefits such as the existence of toric degenerations, \textit{cf.}~\cite{anderson}, and completely integrable systems, \textit{cf.}~\cite{HaradaKaveh}, to name but a few. In this article,  we take a few steps away from the situation where every participant is toric: we consider valuation semigroups associated with torus-invariant divisors on toric surfaces with respect to a \emph{non-toric} valuation. 

 The main idea behind Newton--Okounkov theory is to attach com\-bi\-na\-to\-ri\-al/con\-vex-geometric objects to geometric situations to facilitate their analysis, in other words, to partially replicate the setup of toric geometry in settings without any useful group action. The basis for the theory was developed by Kaveh--Khovanskii \cite{KK12} and Lazarsfeld--Musta\c t\u a  \cite{LM09}, building on earlier work of Okounkov \cite{Ok}, but the subject has seen substantial growth in the last decade. By now applications of Newton--Okounkov theory range from combinatorics and representation theory through birational geometry, \textit{cf.}~\cite{KL_Inf,KL_Pos,KL_Reider,GeomAsp}, to mirror symmetry, \textit{cf.}~\cite{RW}, and geometric quantization in mathematical physics. 

Given a projective variety $\var$ and a divisor $\ds$ on $\var$, Newton--Okounkov theory associates to 
$(\var,\ds)$ a semigroup  $\sg_{\flag}(\ds)$, the valuation semigroup, and a convex body $\nob{\flag}{\ds}$, the Newton--Okounkov body of~$\ds$. Both the valuation semigroup and the Newton--Okounkov body depend, however, on a maximal rank valuation of the function field of $\var$ coming from an admissible flag $\flag$ of subvarieties. 

The Newton--Okounkov body $\nob{\flag}{\ds}$ is an asymptotic version of $\sg_{\flag}(\ds)$ and is, accordingly, a lot easier to determine. Newton--Okounkov bodies on surfaces end up being almost rational polygons; \textit{cf.}~\cite{KLM12}. If the section ring of $\ds$ is finitely generated, then a suitably general flag valuation will yield a rational simplex as its Newton--Okounkov body; \textit{cf.}~\cite{AKL}. In the case of a toric variety $\var$ with torus-invariant $\ds$ and $\flag$, the associated Newton--Okounkov body recovers the moment polytope of the polarized toric variety. 

In this paper, we will focus on the valuation semigroup $\sg_{\flag}(\ds)$, more concretely, on the question of whether or not it is finitely generated. It is a classical fact that $\sg_{\flag}(\ds)$ is often not finitely generated even if~$\var$ is a smooth projective curve. 

It is known that $\sg_{\flag}(\ds)$ is a finitely generated semigroup
if $\var$ is a toric variety, $\ds$ a torus-invariant divisor, and
$\flag$ an admissible flag of torus-invariant subvarieties. We
consider the next open question, namely, the case of toric surfaces
and non-torus-invariant flag valuations. Although the divisorial
geometry of toric surfaces themselves is not particularly complicated,
things get out of control once we start blowing up non-toric
points. Blowing up just one point on a toric surface can lead to surfaces with infinitely
many negative curves on them, as recent research of Castravet--Laface--Tevelev--Ugaglia \cite{castravet2020blown} illustrates. Blowing up many general points quickly leads to notoriously
difficult situations like Nagata's conjecture.

\subsection{Newton--Okounkov bodies}
\label{introNObodies}

Assume that $\var=\TV(\fan)$ is a smooth projective toric variety
of dimension~$n$.
Then ample torus-invariant divisors $\ds$ or their associated
line bundles $\CO_\var(\ds)$
can be understood as lattice polytopes
$\divpol{\ds}$ in the character lattice $\latm \cong\Z^n$ of
the torus $\tor$ acting on $\var$. Using this description, the starting fan
$\fan$ can be recovered as the normal fan of these polytopes.

Moreover, it is a well-known feature of the toric theory that the vector space
$\kG(\var,\CO_\var(\ds))$ has the set of lattice points $\divpol{\ds}\cap \latm $ as its
distinguished basis. That is, the polytope $\divpol{\ds}$ gives 
$\kG(\var,\CO_\var(\ds))$ not just a dimension but also a shape.

In \cite{LM09} and \cite{KK12}, this concept was generalized to
arbitrary projective varieties
$X$ (still of dimension~$n$). If we are given a so-called \emph{admissible flag}
$$\flag: \var=Y_0 \supseteq Y_1 \supseteq \ldots \supseteq Y_n $$
of nested (irreducible) varieties with $\dim Y_i=n-i$ that are
smooth in the special point
$Y_n$, then for every ample divisor $\ds$, there is an associated
convex body $\nob{\flag}{\ds}$  
(``Newton--Okounkov body'') in $\R^n$ reflecting many properties of $\ds$; 
see Section~\ref{dfnnobody} for details.

Note that $\Z^n$ has ceased to be a character
lattice because there is no longer a torus around. While the
Newton--Okounkov body
$\nob{\flag}{\ds}$  depends not on $\ds$ but only on its numerical class,
the dependence on the chosen flag is striking.

If, for instance, $\var$ is toric as at the very beginning, then 
the construction of $\nob{\flag}{\ds}$ recovers the correspondence
between divisors and polytopes we  mentioned above. But 
to make this true requires a {\em toric flag\,}; \textit{i.e.},
all subvarieties
$Y_i$ are supposed to be orbit closures. 

The fact that toric varieties with toric flags lead to well-known
polytopes is not a one-way street. In fact, if, for a general variety
$\var$, the semigroups $\sg_{\flag}(\ds)$ are finitely generated, then
they provide a toric degeneration of $\var$. This was shown in
\cite{anderson}.  Observe that in general the semigroup
$\sg_{\flag}(\ds)$ being finitely generated is much stronger than
$\nob{\flag}{\ds}$ being polyhedral.  This finite generation of the
valuation semigroup $\sg_{\flag}(\ds)$ is the main point of this
paper.

\subsection{Results}
\label{introResults}

In \cite{wellPoised} the finite generation was  shown for complexity one $\tor$-varieties with
toric flags. Note that this implies the toric case.
Whenever one is only interested in the Newton--Okounkov body (instead of the semigroup), 
there are more results:
The most general one solves the question for surfaces using Zariski decomposition; \textit{cf.}~\cite[Theorem 6.4]{LM09}. 
In particular, the Newton--Okounkov bodies are polyhedral in this case.   
Specializing this situation, \cite{HKW20} has provided
an explicit combinatorial description of $\nob{\flag}{\ds}$ 
for toric surfaces with certain \emph{non-toric} flags.

In the present paper, partially inspired by \cite{cohomLB}, we consider the very same setup of toric surfaces,
namely with the following admissible flag: 
$$\flag: \var=Y_0 \supseteq Y_1  \supseteq Y_2,$$
where $Y_1$ is the closure of a one-parameter subgroup of the torus,
which is non-torus-invariant and 
$Y_2$ is a general smooth point. Then we prove that, in dependence on $Y_1$,
the semigroups $\sg_{\flag}(\ds)$
can be both finitely generated and not finitely generated.

The  main result of the paper (Theorem~\ref{thmfgsg}) comes from understanding the relationship between the Newton--Okounkov body of $\ds$ and the Newton polygon associated with the \emph{non-toric} flag curve $Y_1$. The significance of our contribution lies in the fact that we infer the finite generation of valuation semigroups from asymptotic/convex-geometric data and provide a combinatorial criterion. 

In order to state our Theorem, we introduce a bit of terminology. For our flag $\flag$, we consider a non-torus-invariant curve $Y_1$ given as the closure of the one-parameter subgroup determined by a primitive vector $\direcn \in \latn$; our flag point is going to be $0\otimes_{\Z} 1$ on the torus $ \T \cong \N\otimes_{\Z} \C^{\times}$. For a strongly convex cone $\con \subseteq \latn_\R$ and a lattice point $\vect \in \interior(\con)\cap \latn$, we say that $\vect$ is \emph{strongly decomposable in $\con$} if $\vect=\vect'+\vect''$ for suitable $\vect',\vect'' \in \interior(\con)\cap \latn $. Given a divisor $\ds$ on $\var$, we construct strongly convex cones $\conbig$ and $\consmall$ associated with $\direcn$ that are spanned by certain rays of the fan of $\var$ (see Definition~\ref{dfndline}).  With this said, our main result goes as follows.

\begin{introthm}[Theorem~\ref{thmfgsg}]
Let $\var$ be a smooth toric surface associated with a fan 
$\fan$ and  $\ds$ an ample divisor on $\var$. 
The valuation semigroup $\sg_{\flag}(\ds)$ is finitely generated 
 if and only if 
 $\direcn$ is not strongly decomposable in $\conbig$ and
  $-\direcn$ is not strongly decomposable in $\consmall$.
\end{introthm}

To illustrate the combinatorial content, Figure~\ref{introfig7gon} pictures the situation in the case of the 7-gon of \cite{castravet2020blown}, where the blow-up surface $\var=\toric(\fan)$ accomodates infinitely many negative curves (\textit{cf.} Figure~\ref{fig7gon} for a complete picture).

\begin{figure}[h!]
\centering
     \begin{subfigure}[h]{0.45\textwidth}
         \centering
        \includegraphics[]{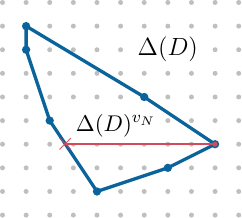}
             \caption{}
         \end{subfigure}
     \begin{subfigure}[h]{0.45\textwidth}
  \centering
        \includegraphics[]{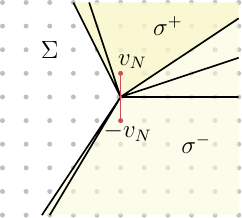}
                \caption{}
     \end{subfigure}
       \caption{{\bf The {good} {\boldmath$7$}-gon {\boldmath$\divpol{\ds}$}.} 
       (A) The polytope $\divpol{\ds}$ and
       the rational line segment $\dline$.
       (B) The normal fan $\fan$ of $\divpol{\ds}$ 
       together with the two cones $\conbig \ni \direcn$  and 
       $\consmall \ni - \direcn$.}
\label{introfig7gon}
\end{figure} 

As an application of the above theory, we construct in Example~\ref{final:ex} a lattice polygon with a strong non-finite-generation property. To be more concrete, we look at the ample divisor $\ds$ associated with the polytope $\divpol{\ds}$ given in Figure~\ref{finalex}\subref{finalex1} on the toric variety $\var = \toric(\fan)$ corresponding to the fan $\fan$ in Figure~\ref{finalex}\subref{finalex2}.  Example~\ref{final:ex} shows that the semigroup $\sg_{\flag}(\ds)$ will not be finitely generated for any $\direcn \in N$ we pick.

\section{Notation and preliminaries}
\label{preNot}

Let $X$ be a two-dimensional smooth projective variety, and assume that we are given an admissible flag
$$\flag: \var=Y_0 \supseteq Y_1  \supseteq Y_2, $$
as in Section~\ref{introNObodies},  and an ample divisor $\ds$  on $\var$. 
 
\subsection{Newton--Okounkov bodies and valuation semigroups} 
\label{dfnnobody}\label{valuations}

Following \cite{LM09}, we obtain a rank two
valuation-like function
(or, equivalently, a rank two valuation of the function field of $X$,
see \cite{KMR_Abstract})
$$\val_{\flag}\colon\kG(X,\CO_{\var}(\ds))\setminus\{0\}\lra\Z^2$$
as follows:
Let $f$ be an equation for $Y_1$ near $Y_2$. 
For a non-trivial section $s\in\kG(X,\CL)$ of a line bundle
$\CL$, \textit{e.g.}, $\CL=\CO_{\var}(\ds)$, we define
$$\val_{\flag}(s)=(\val_1(s),\val_2(s)) \colonequals
\left(\ord_{Y_1}(s), \ord_{Y_2}\left(\til{s}|_{Y_1}\right)\right),$$
where 
$\til{s} \colonequals s/f^{\val_1(s)}$ is a section of 
$\CL\otimes\CO_{\var}(-\val_{1}(s)\cdot Y_{1})$.

The valuation semigroup $\sg_{\flag}(\ds) $ of $\ds$ (with respect to the flag $\flag$) is
defined as
\begin{align*}
\sg_{\flag}(\ds) \colonequals \{(\el, \val_{\flag} (s))\, \vert \, 
s \in \kG(\var,\CO_\var(\el \ds))\setminus \{0\} , \el \in \N \} \subseteq \N^{3}.
\end{align*}
The Newton--Okounkov body of $\ds$ (with respect to the flag $\flag$) is
defined as the set
\begin{align*} 
\nob{\flag}{\ds} \colonequals \overline{\bigcup_{\el \geq 1}{ \frac{1}{\el} \left\{ \val_{\flag} (s) \, \vert \, s \in \kG(\var,\CO_\var(\el \ds))\setminus \{0\} ,  \el \in \N \right\}}} \subseteq \R^2.
\end{align*}
For Newton--Okounkov theory on surfaces, see \cite{KL_surf}. 

\subsection{Toric setup}
\label{setup}

Let $\latn\cong\Z^2$ be a two-dimensional lattice with dual lattice $\latm$ and $\fan$ a 
smooth complete fan associated with the toric surface $\var=\toric(\fan)$. We may assume that our ample divisor $\ds$ is toric, hence is represented by a polytope $\divpol{\ds}$. Recall that $\tor=\spec{(\C[\latm])}$ is  our torus acting on $\var$; hence $\latm$ becomes its character lattice and $\latn$ the associated lattice of one-parameter subgroups. 
 
Next, fix an admissible flag 
$\flag: \var \supseteq Y_1  \supseteq Y_2$ as follows:
Choose a primitive element $\direcn \in \latn$. The embedding
$$\iota\colon \ko{\latn} \colonequals \Z\direcn  \longhookrightarrow \latn$$
induces a map of fans, hence a toric map $\iota\colon \PP^1 \to \var=\toric(\fan)$.
Set $Y_1 \colonequals
\iota(\PP^1)$ and $Y_2 \colonequals \iota(1 \in \PP^1)$.

Within the torus $\tor$, the curve $\curv \colonequals Y_1$ is given by the binomial equation
$\binomf \colonequals x^{\direcm}-1$ with $\direcm \in \latm$ being one of the two primitive elements of
$\direcnbot \subset \latm_{\R}$. The associated Newton polytope 
$\newpolv \colonequals \newton(\binomf)$
is the line segment $[0,\direcm]$ connecting $0$ and $\direcm$ in $\latm_{\R}$. This way 
$\iota\colon \PP^1 \to \curv$ is the normalization map.

For a toric line bundle $\CL$ on $\var$, the pullback $\iota^* \colon \Gamma(X,\CL) \to \Gamma(\PP^1, \iota^*\CL)$ corresponds to the projection $$\pi \colon \latm \longrightarrow \latm/\Z\direcm \equalscolon \ko{\latm} ,$$
which we will identify with $\direcn \colon \latm \to \Z \cong \ko{\latm}$.

Observe that this almost fits the setup of~\cite{wellPoised} as $Y_1$ is invariant under the codimension one torus given by $\direcn$. We only deviate from~\cite{wellPoised} by choosing a non-invariant flag point $Y_2$.

\subsection{Torifying the curve}
\label{defval}
Note that  $\curv$ is also  a prime (Cartier) divisor on $X$, which properly
intersects all torus-invariant curves, \textit{i.e.}, all
boundary curves of $\var$. Therefore, $\curv$ is nef and 
$\curvbar\colonequals \curv - \div(\binomf) $
is its $\tor$-invariant representative. 

\begin{lemma}
\label{shapenefpoly}
The polygon $\nefpolv\colonequals\divpol{  \curvbar}$ is given by
$$\nefpolv= \left\{ \op \in \latm_\R \, \vert \, 
\langle \op, \ray \rangle \geq
\min\{0,\langle\direcm,\ray\rangle\} 
\text{\rm\ for all } 
\ray \in \fan(1) \right\},$$
where the rays $\ray\in\fan(1)$ are identified with their first $($hence
primitive$)$ lattice points.
\end{lemma}

\begin{proof}
The prime divisor $\ds_\ray=\ko{\orb(\ray)}$ associated with a ray $\ray\in\fan(1)$
appears in the torus-invariant Weil divisor $\curvbar$ as often as it
does in the (non-equivariant) principal divisor $-\div(\binomf)$.
Thus,
$$ \nefpolv= \left\{ \op \in \latm_\R \, \vert \, 
\langle \op, \ray \rangle \geq
\ord_\rho (\binomf)
\text{ for all } \ray \in \fan(1) \right\},$$
and it remains to discuss
$\ord_\rho (\binomf)$.
If $\ray\neq\pm\direcn$, then the $\ray$-orders of the two summands of
$\binomf=x^{\direcm}-1$ are different, hence
$$\ord_\ray(f)=\min\{\ord_\ray(1),\ord_\ray(x^{\direcm})\}=
\min\{0,\langle\direcm,\ray\rangle\} .$$ 
If $\ray=\pm\direcn$, then  the situation looks locally
like $\ord_y(x-1)=0$ in $\A^2$.
\end{proof}

See  \cite[Proposition 3.1]{HKW20} for a different proof.

\begin{figure}[b]
     \centering
          \hspace{-1.5cm}
     \begin{subfigure}[b]{0.45\textwidth}
         \centering
               \includegraphics[]{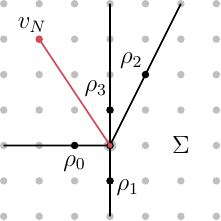}
         \caption{}
         \label{figex11}
     \end{subfigure}
      \hspace{-0.5cm}
     \begin{subfigure}[b]{0.45\textwidth}
         \centering
        \includegraphics[]{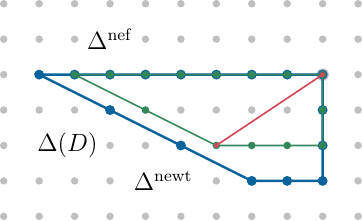}
                 \caption{}
         \label{figex12}
     \end{subfigure}
       \caption{{\bf Input data.} (A) The fan $\fan$ together with 
       its set of rays $\{  \rho_i \, \vert \, 0 \leq i \leq 3 \}$ and a primitive element $\direcn=(-2,3) \in \latn$.
       (B) The polytope $\divpol{\ds}$ corresponding to the ample divisor $\ds= 8\ds_2+3 \ds_3$ on $\var=\TV(\fan)$,  
       the Newton polytope $\newpolv$ given by $\direcm=[-3,-2] \in \latm$, and the polytope $\nefpolv$ corresponding to $\curvbar$.
}
\label{figex1}
\end{figure}

\begin{example} \label{ex1}
Let $\var=\TV(\fan)$ be the toric surface associated with the fan 
$\fan$, where 
 $\fan(1)=\{  \rho_i \, \vert \, 0 \leq i \leq 3 \}$ with
$$\rho_0= (-1,0),\; \rho_1=(0,-1),\; \rho_2=(1,2), \text{ and } \rho_3=(0,1).$$
As in Lemma~\ref{shapenefpoly}, we identify the rays $\rho_i$ with their generating lattice points (\textit{cf.} Figure~\ref{figex1}\subref{figex11}).  We denote by $\ds_i\colonequals\ko{\orb(\rho_i)}$ ($0 \leq i \leq 3$) the toric prime divisors on $\var$.  Then, $\ds\colonequals 8\ds_2+3 \ds_3$ is an ample divisor on $\var$, which corresponds to the polytope
$$\divpol{\ds} = \conv([0,0],[-8,0],[-2,-3],[0,-3]) .$$
We take $\curv = \{(t^{-2},t^3)\kst t\in\C \}$ as our
curve for the non-toric flag. This means that
$\direcn = (-2,3) \in N$ and $\direcm= [-3,-2] \in \latm$, hence
$$\newpolv = \conv([0,0],\,[-3,-2]).$$  
Since the boundary part of $\div(\binomf)$
equals $-7\ds_2-2\ds_3$, we obtain $\curvbar=7\ds_2+2\ds_3$ 
with nef polytope (\textit{cf.} Figure~\ref{figex1}\subref{figex12})
$$ \nefpolv= \conv([0,0],[-7,0],[-3,-2],[0,-2]) .$$ 
\end{example} 

\subsection{An alternative view on {\boldmath $\nefpolv$}}
\label{plusminus}

Beside the explicit description of Lemma~\ref{shapenefpoly}, it is possible to describe the shape of $\nefpolv$ in the following more combinatorial way.
The relation $\direcm \in \direcnbot$ among our curve parameters means
$$\langle 0, \direcn \rangle = \langle \direcm, \direcn \rangle = 0;$$
\textit{i.e.}, $\newpolv=\conv(0,\direcm)$ is contained in the level set
$[\direcn=0]$.

We denote by $\vplus, \vminus \in \latm$ the vertices of $\divpol{\ds}$, where 
$\langle \divpol{\ds}, \direcn \rangle$ 
admits its extremal values
(\textit{cf.} Figure~\ref{figex2}\subref{figex21}).
Moreover, we define
$\coneplus, \coneminus$ to be the two-dimensional 
cones generated by the two edges of $\divpol{\ds}$ 
that contain the vertices $\vplus$ and $\vminus$, respectively.

We take the line segment $\newpolv$ and fit it inside
the cone $\coneplus$ until it hits both rays of this cone. In this
way, we construct a lattice triangle $\nablaplus$ with base $\newpolv$
and top vertex $\vplus$. We construct $\nablaminus$ (\textit{cf.}
Figure~\ref{figex2}\subref{figex21}) in a similar way.  In other
words, the cones $\coneplus$ and $\coneminus$ are cut off along
$\direcn$-constant lines producing edges of $\nablaplus$ and
$\nablaminus$, respectively, of lattice length one.  Note that both
cut lines are parallel translates. Gluing $\nablaplus$ and
$\nablaminus$ along $\newpolv$ (\textit{cf.}
Figure~\ref{figex2}\subref{figex22}) yields
$$\nefpolv = \nablaplus \cup_{\newpolv} \nablaminus.$$

Note that $\nefpolv\supseteq \newpolv$ is the smallest polytope
containing $\newpolv$ and having $\Sigma$ as a refinement of its
normal fan.  Actually, either $\nefpolv$ is a quadrangle with
$\newpolv$ serving as one of its diagonals, or it is a triangle with
$\newpolv$ as a side.

\begin{figure}[h!]
     \centering
     \begin{subfigure}[b]{0.45\textwidth}
         \centering
          \includegraphics[]{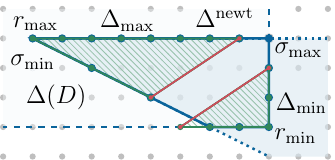}
                   \caption{}
              \label{figex21}
     \end{subfigure}
\hspace{-0.5cm}
     \begin{subfigure}[b]{0.45\textwidth}
         \centering
        \includegraphics[]{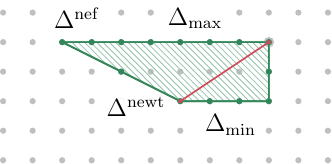}
         \caption{}
         \label{figex22}
     \end{subfigure}
       \caption{ {\bf Alternative view on {\boldmath $\nefpolv$}.}
           (A) The two 
           cones  $\coneplus$, $\coneminus$ and polytopes
            $\nablaplus$, $\nablaminus$  constructed as in Section~\ref{plusminus} with 
            $\direcm = [-3,-2]\in \latm$.
                    (B) Gluing  $\nablaplus$ and $\nablaminus$  along  $\newpolv$ 
 yields $\nefpolv$.
}
\label{figex2}
\end{figure}

\section{Valuation semigroups associated with non-toric flags}
\label{sgconstruction}

In  this section, we determine the valuation semigroup $\sg_{\flag}(\ds)$ associated with an ample  (Cartier) divisor $\ds$ and a \emph{non-toric} flag $\flag$ as a subset of $\N^3$. The main result is Theorem~\ref{theoshapesg}, where the abstract semigroup $\sg_{\flag}(\ds)$  is described in terms of lattice points coming from a polyhedral construction  in $\latm$.  

Let us fix $\el \geq 1$ and $\ka \geq 0$.
The space of sections 
$s \in \kG(\var,\,\CO_{\var}(\el \ds))$ which have vanishing order at least $\ka$ along $\curv$
is the image of  
$$ \kG\left(\var,\,\CO_{\var}(\el \ds-\ka\, \curv)\right) \overset{f^\ka \cdot}{\longrightarrow}  
\kG\left(\var,\,\CO_{\var}(\el \ds)\right), \quad \til{s} \longmapsto s .$$
We set $\CLlk\colonequals\CO_{\var}(\el\ds-\ka\,\curv)$ and 
$\vau\colonequals\kG\big(\var,\,\CLlk \big)$.
If we have $\ord_{\curv}(s) = \ka$, then $\ord_{Y_2}(\til{s}\vert_{\curv}) =
\ord_{1 \in \PP^1}(\iota^*(\til{s}))$ using the pullback via
$\iota\colon \PP^1 \to \var$.

\subsection{Return to toric geometry}
\label{makeToric}
Our goal is to understand the restriction of global sections via toric geometry. Therefore, we implement two changes. First, we will shift the linear series of the flag curve, which enables us to replace some of the line bundles we study with torus-invariant ones. Second, we will normalize the restriction.

We are going to use $\curvbar= \curv - \div(\binomf)$
from Section~\ref{defval}. In terms of the associated sheaves, this means
$$\CO_{\var}(  \curvbar)=\binomf\cdot\CO_{\var}(\curv),$$
where $\binomf$ is
 the equation $x^{\direcm}-1$ of $ \curv$ mentioned earlier.
This leads to the possibility of replacing
$\CLlk$ by the isomorphic, but torus-invariant, line bundle
$$\CLlktor \colonequals\CO_{\var}(\el\ds-\ka\,  \curvbar)=\binomf^{-\ka}\cdot \CLlk
\;\subseteq \C(\var) .$$
Accordingly, we set
$$\vautor \colonequals\kG\big(\var,\,\CLlktor\big) $$
and replace $\til{s}\in \vau$ with $s'\colonequals\binomf^{-\ka}\cdot \til{s}\in \vautor$. 

Recall that the nef invertible sheaves $\CO_{\var}(\el\ds)$ and $\CO_{\var}(\ka\,\curvbar)$
correspond to the polytopes $\el\divpol{\ds}$ and
$\ka \nefpolv$, respectively. 
This implies that
$\vautor$ has a monomial base provided by
\begin{align*}
\thet&\colonequals(\el\divpol{\ds}:\ka \nefpolv) \colonequals \{\op \in \latm_\R\kst\op+\ka \nefpolv \subseteq \el\divpol{\ds}\} \\
&\phantom{:}=(\el\divpol{\ds}:\ka \newpolv) \colonequals \{\op \in \latm_\R\kst\op+\ka \newpolv \subseteq \el\divpol{\ds}\} ; 
\end{align*}
see, \textit{e.g.}, the $(i=0)$-case of \cite[Theorem 2]{cohomLB}.

\begin{remark}
In toric geometry, there is a well-known way to associate to every divisor $D$ a polytope $P_D$ reflecting the global sections of $\CO_{\var}(D)$. If (and only if) $D$ is nef, this mapping $D \mapsto P_D$ allows one to recover $D$, \textit{i.e.}, creates a one-to-one correspondence.  Therefore, we introduce another polyhedral gadget $\Delta(D)$ providing a complete characterization for more general divisors. If $D$ is nef, then $\Delta(D)=P_D$ coincides with the former construction. However, for general divisors $D=A-B$ ($A,B$ nef), $\Delta(D)$ becomes a virtual polytope, \textit{i.e.}, a formal difference of two true polyhedra that is Minkowski-additive in its arguments. More concretely, we have $\Delta(A-B)=P_A-P_B$.  The global section polytope $P_D$ for arbitrary divisors $D=A-B$ can be recovered as $P_D = (\Delta(A):\Delta(B))$. If $D$ was nef, then we could choose $B$ as the zero divisor, \textit{i.e.}, $D=A$, and obtain again $P_A=\Delta(A)$.
\end{remark}

\begin{figure}[b]
     \centering
        \includegraphics[]{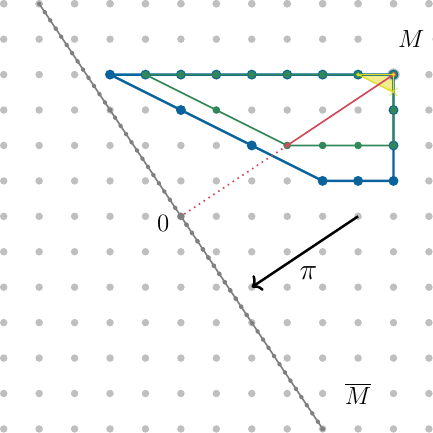}
       \caption{{\bf Projection map.} $\pi\colon\latm \surj\ko{\latm} = \latm/\Z\,[-3,-2]$, $[x,y] \mapsto -2x+3y$.} 
\label{figex3}
\end{figure}

\begin{example} \label{ex2}
Continuing Example~\ref{ex1}, we obtain 
\begin{align*}
\thet&= 
\left(\el\divpol{\ds}:\ka \nefpolv\right)\\
&=\{\op \in \latm_\R \,\vert\, 
\langle \op , (-1,0) \rangle \geq 0 , 
\langle \op, (0,-1) \rangle \geq 0 , \\
&\phantom{= \, \, \, \, \op \in \latm_\R \,\vert\, }\langle \op, (1,2) \rangle \geq 7\ka-8 \el,
\langle \op, (0,1) \rangle \geq 2\ka-3 \el\}.
\end{align*}
In particular, $\Theta(1,1) = \conv([0,0],[-1,0],[0,-1/2])$
(\textit{cf.} the yellow polytope in Figure~\ref{figex3}).
The projection  looks like
$\pi\colon M\surj\ko{\latm}= \latm/ \Z \direcm$.
This map (\textit{cf.} Figure~\ref{figex3}) can be identified with $\direcn\colon\latm\to\Z$, that is,  
$$\pi\colon [x,y] \longmapsto -2x+3y .$$
\end{example}

\subsection{An alternative view on {\boldmath $\thet$}} 
\label{altthet}
In general, $\fan$ is not the normal fan of $\thet$ as it was of
$\divpol{\ds}$.
Geometrically, this means 
that $\thet$ does,  in general, \emph{not} encode a nef Cartier divisor on $\var$.
While $\thet$ is defined as some kind of a difference of polytopes,  it is in general not true that the inclusions
$$\thet + \ka \nefpolv \subseteq  \el\divpol{\ds}\quad  \text{and}\quad  
\thet + \ka \newpolv \subseteq  \el\divpol{\ds}  $$ 
become  equalities (\textit{cf.} Example~\ref{ex4}).
We present a suggestion on how to overcome this. 

Recall from Section~\ref{plusminus} that we had denoted by $\vplus, \vminus \in \latm$ the vertices of $ \divpol{ \ds}$ where $\langle \divpol{ \ds}, \direcn \rangle$ attains its extremal values.  Similarly, we denote by $\vplus'(\el,\ka), \vminus'(\el,\ka) \in \latm_\R$ the $\direcn$-extremal vertices of $\thet$.  The latter lead to the line segments $\vplus'(\el,\ka) + \ka \newpolv$ and $\vminus'(\el,\ka)+ \ka \newpolv$, which cut the polytope $\el \divpol{ \ds}$ into three subpolytopes which we call $\quadplus(\el,\ka)$, $\thetcutelka$, and $\quadminus(\el,\ka)$.

More concretely, here is how we obtain $\quadplus(\el,\ka) $ and $\quadminus(\el,\ka) $: We take the line segment $\ka \newpolv$ and fit it inside the polytope $\el \divpol{ \ds}$ until it hits the boundary twice. This way, we construct the lattice polygon (not necessarily a triangle) $\quadplus(\el,\ka) $ such that $\ka \newpolv$ is one of its edges and $\vplus(\el,\ka) $ is one of its vertices. In a similar way, we construct $\quadminus(\el,\ka) $ using $\vminus(\el,\ka) $.

As $\el\divpol{\ds}$ before, the  polytope $\thetcutelka$ just defined  still fulfills the equality
$$\thet = \left(\el \thetcutelka :\ka \newpolv\right);$$
however, now we also  have the equality
\label{cutpol}
$$\thet + \ka \newpolv =  \el \thetcutelka.$$ 

 \begin{remark}
After this point,  we will use  the shorter notation
$\vplus'= \vplus'(\el,\ka) $,
$\vminus'=\vminus'(\el,\ka)$,
$\quadplus=\quadplus(\el,\ka)$,
$\thetcut=\thetcutelka$, 
$\quadminus= \quadminus(\el,\ka)$.
Nevertheless, one should keep in mind that all of these quantities depend on $\el$, $k$. 
\end{remark}

\begin{figure}[h!]
     \centering
     \begin{subfigure}[b]{0.45\textwidth}
         \centering
               \includegraphics[]{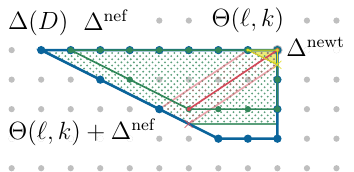}
         \caption{}
         \label{figex41}
     \end{subfigure}
\hspace{-0.5cm}
     \begin{subfigure}[b]{0.45\textwidth}
         \centering
         \includegraphics[]{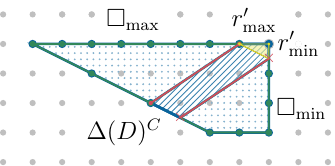}
                  \caption{}
         \label{figex42}
     \end{subfigure}
              \caption{{\bf Alternative view on {\boldmath $\thet$}.}
           (A) The Minkowski sum $\thet + \ka \nefpolv$ and $\thet + \ka \newpolv$.
                    (B) The cut of $\divpol{\ds}$ along $\vplus'+\ka \newpolv$ and $\vminus'+\ka \newpolv$ 
                    into $\quadplus$, $\quadminus$, and $\thetcut$
                    with $\thetcut = \thet +  \ka \newpolv$.
}
\label{figex4}
\end{figure}

\begin{example} \label{ex4}
Continuing Example~\ref{ex2}, 
Figure~\ref{figex4}\subref{figex41} shows that the inclusions
$$\thet +\ka\nefpolv  = \conv([0, 0],[-8, 0],[-3, -5/2],[0, -5/2])
\subsetneq  \divpol{\ds} $$
and 
$$\thet + \ka \newpolv = 
\conv([0, 0],[-1, 0],[-4, -2],[-3, -5/2],[0, -1/2])
\subsetneq  \divpol{\ds}$$ 
are strict in general for $(\el,\ka)=(1,1)$.
We cut the  polytope  $ \divpol{ \ds}$ along the line segments
$\vplus' +\ka \newpolv$ and $\vminus' +\ka \newpolv$  
into the subpolytopes 
$\quadplus =  \conv([-1, 0],[-8,0],[-4, -2])$, 
$\thetcut = \thet + \ka \newpolv$, 
and
$\quadminus =  \conv([0,-1/2],[0, -3],[-2, -3],[-3, -5/2]),$
where
$\vplus' =[-1,0] $ and $\vminus' =[0,-1/2]$ (\textit{cf.} Figure~\ref{figex4}\subref{figex42}).
\end{example}

\subsection{Projections of polytopes}
\label{projpoly}

We start by pulling back the sheaf $\CLlktor$. To this end, we define  
$$\debarelka \colonequals \el\cdot \wid{\direcn}(\divpol{\ds})-\ka \cdot \wid{\direcn}\left( \nefpolv\right),$$ where $ \wid{\direcn}(\cdot)$ denotes the lattice width of a polytope with respect to the linear functional $\direcn \in \latn$; \textit{i.e.}, if $\poly \subseteq \latm_\R$ is a polytope, then $\wid{\direcn}(\poly) \colonequals \max_{\op,\op' \in \poly} $ ${\vert \langle \op,\direcn \rangle-\langle \op',\direcn \rangle \vert}$.  Note that this equals the length of the line segment $\ko{\poly}\colonequals \pi(\poly)$; \textit{i.e.},
$$\debarelka =\el \cdot \length\left(\ko{\divpol{\ds}}\right)-\ka \cdot \length\left(\ko{ \nefpolv}\right).$$

\begin{proposition}
\label{prop-projPol0}
The pullback
$\iota^* \CLlktor$ is a line bundle on $\PP^1$ of degree
$\debarelka$.
\end{proposition}

\begin{proof}
We obtain
\begin{align*}\pushQED{\qed}
\iota^* \CLlktor
=  \iota^* \CO_{\var}\left(\el \ds - \ka \curvbar\right) 
&= \iota^* \CO_{\var}\left(\el \divpol{\ds} - \ka \nefpolv\right) \\
 &= \CO_{\PP^1}\left(\el\cdot(\ds  . \curvbar) - \ka\cdot (\curvbar\, .\curvbar)\right) \\
&= \CO_{\PP^1}\left(\el\cdot \wid{\direcn}(\divpol{\ds})-\ka \cdot \wid{\direcn}\left( \nefpolv\right)\right) .\qedhere\popQED
\end{align*}
\renewcommand{\qed}{}
\end{proof}

\begin{remark}
Altogether, this yields the sequence of  inclusions
$$
\pi(\thet\cap \latm)\subseteq \ko{\thet} =\ko{\left(\el{\divpol{\ds}}: \ka { \nefpolv}\right)}  \subseteq \left(\el\ko{\divpol{\ds}}: \ka \ko{ \nefpolv}\right) =: \kothet,$$
which might be strict, where $\ko{\Delta} = \pi(\Delta)$ is 
the projection of any polytope $\Delta \subseteq \latm$ along $\direcm$. 
\end{remark}

\begin{example} \label{ex5}
Continuing Example~\ref{ex4}, let us   fix $(\el, \ka)=(1,1)$. Then  
the projected polytopes along $\pi$ are 
\begin{align*}
\ko{\divpol{\ds}}
&= \conv([-9],[16]),
\quad
\ko{\nefpolv}
= \conv([-6],[14]),\\
\ko{\thet}
&= \conv([-3/2],[2]),
\quad \text{and}\quad 
\kothet 
=  \conv([-3],[2]) ,
\end{align*}
where 
$\debarelka = 25 - 20 = 5$ (\textit{cf.} Figure~\ref{figex5}).
\end{example}

\begin{figure}[h!]
     \centering
        \includegraphics[scale=.8]{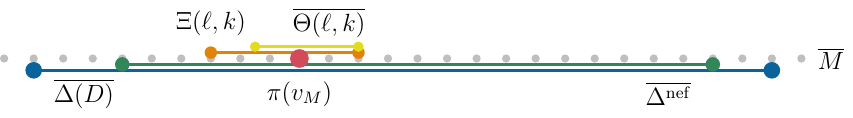}
       \caption{{\bf Projected polytopes.} 
}
\label{figex5}
\end{figure}

Recall that the (torus-equivariant) global sections of $\CLlktor$ are encoded by the elements of $\thet\cap \latm$. Under this identification, their pullbacks via $\iota^*$ are given by their images under $\pi$.  Denote their number by
 $$\eebarelka\colonequals\# \pi(\thet\cap \latm).$$
Summarizing what we have done so far, we obtain the following. 

\begin{proposition}
\label{prop-projPol}
The pullback
$\iota^* \CLlktor=\CO_{\PP^1}(\kothet)$ is a line bundle on $\PP^1$\! of degree
$\debarelka$.
Its global sections correspond to the elements of\,
$\kothet\cap\ko{\latm}$. Under this identification, the
subspace $\iota^* \vautor$ coincides with $\pi(\thet\cap \latm)$.
In particular,
$$\dim(\iota^* \vautor)=\eebarelka .$$
\end{proposition}

\begin{proof}
Proposition~\ref{prop-projPol0} yields the degree $\debarelka$ of the pullback of $\CLlktor$.  It remains to show that this sheaf is precisely given via $\kothet$ as $\CO_{\PP^1}(\kothet)$, which is slightly finer information.  The statement holds since if $\poly \subseteq \latm_\R$ is a nef polytope (\textit{e.g.}, $\el \divpol{\ds}$ or $\ka \nefpolv$), then
$$\iota^* \CO_{\var}(\poly)= \CO_{\PP^1}(\pi(\poly))= \CO_{\PP^1}\left(\ko{\poly}\right).$$
This claim is valid for any toric map and does not  depend on having 
$\PP^1$ as a target.
\end{proof}

\begin{example} \label{ex6}
Continuing Example~\ref{ex5}, we obtain $\thet \cap \latm = \{[0,0],[-1,0]\} $ and therefore $\pi(\thet \cap \latm )= \{0,2\}$; \textit{i.e.}, $\eebarelka=2$ for $(\el,\ka)=(1,1)$.
\end{example}

\subsection{Shape of the semigroup}
\label{sgshape}

As we did before, let us fix a pair $(\el,\ka)$.  We know from Section~\ref{valuations} that we are supposed to collect the values $\ord_{Y_2}(\til{s}|_{\curv})$ for all possible sections $\til{s}$, where $Y_2=\{1\}$ is a smooth point on $\curv$.  In Section~\ref{makeToric}, we have transferred this setup to $\ord_{1 \in \PP^1}(\iota^*s')$, where $s'$ runs through all global sections represented by the polytope $\thet \subseteq \latm_\R $.

Proposition~\ref{prop-projPol} implies that the pullbacks $\iota^*s'$ run through all $\eebarelka$ elements of $\pi(\thet \cap{\latm}) \subseteq \ko{\latm} \cong \Z$.  Each element of $\Z$ represents a rational monomial function on $\PP^1$.  We are supposed to find the orders of vanishing at $1 \in \PP^1$ of all of their linear combinations.

\begin{lemma}
\label{lem-Vandermonde}
Let $Z\subset\Z$ be a finite subset with $\ee$ elements leading to the $\ee$-dimensional 
vector space
$$\C[Z]\colonequals\{f\in\C[t,t^{-1}] \, \vert \, \supp(f) \subseteq Z\} .$$
Then $\ord_1\C[Z]=\{0,1,\ldots,\ee-1\} = \ord_c\C[Z]$ for all $c \in \C^* \subseteq \PP^1$.
\end{lemma}

\begin{proof}
Set $Z=\{p_1,\ldots,p_{\ee}\}$. 
For an element $f\in\C[Z]$ with
$f= \lambda_1 \cdot  t^{p_1}+ \ldots + \lambda_e \cdot  t^{p_e}$ and
$d \in \N$, the rows of the matrix $P$ given as
$$ 
{\small
\left(\begin{array}{@{}cccccc@{}}
1 & 1 &  & \ldots & 1\\
p_1 & p_2 & & \ldots & p_{e}\\
p_1(p_1-1) & p_2(p_2-1) &  & \ldots & p_{e}(p_{e}-1)\\
p_1(p_1-1)(p_1-2) & p_2(p_2-1)(p_2-2) & & 
 \ldots & p_{e}(p_{e}-1)(p_{e}-2)\\
\vdots & \vdots &  & \ddots & \vdots\\
\,\,\, p_1\cdot \ldots \cdot (p_1-(d-1)) & p_2\cdot \ldots \cdot (p_2-(d-1))  & & 
 \ldots & p_e\cdot \ldots \cdot (p_e-(d-1)) 
\end{array}\right)
}
$$
encode $f(1)= \lambda_1 \cdot 1+ \ldots + \lambda_e \cdot 1$, $f'(1)= \lambda_1 \cdot p_1 \cdot 1+ \ldots + \lambda_e \cdot p_e \cdot 1$, $f''(1),\ldots,f^{(d)}(1)$.  Let $p$ be an arbitrary variable. Then the linear spaces
$$\textstyle L_1(p)\colonequals\spann_\Q\left\{1,p,p^2,\ldots,p^{\de}\right\}
\subseteq \Q[p]$$
and
$$\textstyle L_2(p)\colonequals\spann_\Q\left\{0! {p\choose 0},1! {p\choose 1}, 2!{p\choose 2},\ldots, d!{p\choose d}\right\}
\subseteq \Q[p] \vspace{1ex}$$
coincide because $L_1(p) \supseteq L_2(p)$ and $\dim(L_1(p))=\dim(L_2(p))$.
In particular, 
there exists an invertible lower triangular matrix $F$ such that $P=F\cdot V$, where 
$ V\colonequals (p_{j}^{i})_{0 \leq i \leq d, 1 \leq j \leq e}$
is the transposed Vandermonde matrix.
If we choose $d=e-1$, the matrices $V$ and thus $P$ are invertible.
Hence the system of linear equations
$$P \cdot (\lambda_1, \ldots, \lambda_e)^T= 0,$$ which is equivalent to the system $f(1)=f'(1)=f''(1)=\ldots=f^{(e-1)}(1)=0$, has the unique solution $f=0$, so $\ord_{1}(f) \ge e$ is impossible.

On the other hand, replacing the equation $f^{(k)}(1)=0$ by $f^{(k)}(1)=1$ ($0 \leq k \leq e-1$) yields an $f$ with $\ord_{1}(f) = k$.
\end{proof}

As a direct consequence, we obtain the following statement. 

\begin{theorem}
\label{theoshapesg}
The valuation semigroup is given as  
$$\sg_{\flag}(\ds) = \left\{\left(\el, \ka,\delta\right) \in \N^3\, \vert \,   0 \leq \delta \leq \eebarelka -1\right\}, $$
and
$\eebarelka=0$ for large $\ka \gg \el$. 
\end{theorem}

\begin{proof}
The definition of the valuation semigroup can be reformulated as
$$\sg_{\flag}(\ds) = \{(\el, \ka,\delta)  \in \N^3 \, \vert \,  s' \in \kG(\var,\CLlktor)\setminus \{0\}, 
\ord_{1 \in \PP^1} (\iota^*s')=\delta \} .$$
Then everything follows from Proposition~\ref{prop-projPol}. 
\end{proof}

\section{Shape of the Newton--Okounkov body}
\label{sec:shapenob}

Building on Section~\ref{sgconstruction}, we determine the 
Newton--Okounkov body $\nob{\flag}{\ds}$ in Theorem~\ref{theonobshape}. 
Consider the assignment
\begin{align*} 
\de(\el,\ka)&\colonequals\wid{\direcn}(\thet) =\length(\pi(\thet)) =
              \length(\ko{\thet}) \in \Q \sqcup \{-\infty\}\, . 
\end{align*}
 \label{eq:d(l,k)}
We extend this definition to all $\el, \ka \in \R_{\ge 0}$ using the
convention $\wid{\direcn}(\emptyset) = -\infty$. 
This becomes necessary when  
$\ka \cdot \nefpolv$ does not fit inside
$ \el\cdot \divpol{\ds}$, which will happen for $\ka \gg \el$.

This should not be confused with $\debarelka$, which was defined on page \pageref{projpoly} as
$$\debarelka =  \el\cdot \wid{\direcn}(\divpol{\ds})-
\ka \cdot \wid{\direcn}( \nefpolv).$$ 
The chain of inclusions at the end of
Section~\ref{projpoly}  gives rise to the  inequalities
$$\eebarelka -1 \leq \de(\el,\ka)  \leq   \debarelka. $$

\begin{example} \label{ex7}
Continuing Example~\ref{ex6}, \textit{i.e.}, $(\el, \ka)=(1,1)$, $\debarelka = 5$, and $\eebarelka=2$, we obtain $\de(\el,\ka)=7/2 $ satisfying the inequalities $1 \leq 7/2 \leq 5$.
\end{example}

Moreover, we observe the following. 

\begin{lemma} \label{asymlemma}
For $q \in \R_{\geq 0}$, the assignment $\de(q)\colonequals \de(\el, \el q)/\el$ does not depend on $\el$.  In particular, $\de(q)= \de(1, q)$.
\end{lemma}

\begin{proof}
The width function is linear in its polyhedral argument.
\end{proof}

Note that the same statement holds true for $\debarelka$ but not for $\eebarelka$.

\bigskip
From now on, we return  to $(\el,\ka) \in \N^2$. 

\begin{theorem} \label{theonobshape}
The Newton--Okounkov body  $\nob{\flag}{\ds}$ coincides with the convex hull of the set
$$\{ [q,t] \in (\R_{\geq 0})^2 \,\vert  \, 0 \leq t \leq  \de(q)\}
= \hspace{-0.3cm} \bigcup_{q \in \R_{\geq 0}, \, d(q) \geq 0}
\hspace{-0.5cm} \conv([q,0], [q,\de(q)]).$$ 
Moreover, 
$\de(q)$ is a decreasing piecewise linear  function
 with $\de(q)=-\infty$ for $q \gg 0$.
\end{theorem}

\begin{proof}
Let $\epsilon >0$ and denote by $A,B$ the vertices of $ \pi(\thet)=\ko{\thet}$.  First assume that the dimension of $\thet$ equals two.  Then the two fibers
$$\pi^{-1}( T ) \cap \thet$$
with $T = A+\epsilon$ or $B-\epsilon$ have positive lengths greater than (or equal to) some $\mu >0$.  In particular, all fibers in between do so as well.  Hence, setting $\lambda = 1/\mu$, the fibers
 $$\pi^{-1}( T' ) \cap \lambda \thet$$
have at least length one for $\lambda(A+\epsilon) \leq T' \leq \lambda(B-\epsilon) $.  If in addition $T' \in \ko{\latm}$, then all of these fibers have to contain lattice points in $\latm$.  Thus, we obtain
 $$\conv(\lambda(A+\epsilon) ,  \lambda(B-\epsilon) ) \subseteq  \pi(\lambda \thet \cap \latm) .$$
 
Keeping $q=\ka/\el$ constant and using Lemma~\ref{asymlemma}, we see that $\eebarelka$ behaves like $\de(\el,\ka)$ asymptotically with respect to dilations.  The result then follows by recalling the fact that Newton--Okounkov bodies are closed.

It remains to consider the pathological case $\dim(\thet[1,q])=1$.  Here, we can approximate $[q,t]$ by $[q-\epsilon,t]$ so that the resulting $\thet[1, q - \epsilon]$ is full-dimensional and $t \le \de(q) \le \de(q-\epsilon)$.  Then the previous argument shows that $[q-\epsilon,t] \in \nob{\flag}{\ds}$. As Newton--Okounkov bodies are closed by definition, $[q,t] \in \nob{\flag}{\ds}$.
\end{proof}

We remark that the case $\dim(\thet[1,q])=1$ from the previous proof requires special $\direcn$ and a unique $q_0=\ka/\el$.  This configuration is characterized by the fact that (a shift of) $q_0 \newpolv$ connects two parallel edges of $\divpol{\ds}$.  Note that $\thet$ is also parallel to these edges.  In contrast to the general case, for $\ka/\el=q_0$, the number $\eebarelka$ behaves asymptotically like $1/g \cdot \de(\el,\ka)$, where
$$g\colonequals \de(\el,\ka)/\length_{\latm}(\thet).$$
Despite that $\eebarelka$ for $\ka/\el=q_0$ does not approach $\de(\el,\ka)$ at all, this does not cause a problem: as we have seen in the proof, for $q \leq q_0$, the general case applies, and for $q > q_0$, we have $\de(q)=- \infty$ anyway.

\begin{example} \label{ex8}
We continue Example~\ref{ex7} and apply  Theorem~\ref{theonobshape}. 
The Newton--Okounkov body  $\nob{\flag}{\ds}$  (\textit{cf.}~Figure~\ref{figex7} and \cite{HKW20})
is given as
$$\nob{\flag}{\ds}= \conv([0,0],[0,25],[8/7,0],[2/3,35/3]).$$%

\begin{figure}[h!]
     \centering
        \includegraphics[]{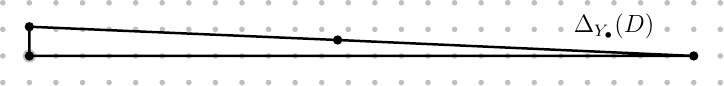}
       \caption{{\bf Newton--Okounkov body {\boldmath$\nob{\flag}{\ds}$}} 
       with flipped coordinates.}
\label{figex7}
\end{figure}

This example already gives an instance of a vertex that does not lift to the semigroup (\textit{cf.} Definition~\ref{def-liftPoints}) when building the Newton--Okounkov body in question.  Let us consider the vertex $[\frac{2}{3},\frac{35}{3}]$ and fix $\el=3$, $\ka=2$.  The respective polyhedra $3\divpol{\ds}$, $2\newpolv$, $2\nefpolv$, and $\thet=(3\divpol{\ds}: 2 \nefpolv)$ are pictured in Figure~\ref{figex8}.  To hit the vertex $[\frac{2}{3},\frac{35}{3}]$, the value of $\overline{\ee(3,2)}$ would have to coincide with $\de(3,2)=35$.  However, we only obtain $\overline{\ee(3,2)}=\#\big(\pi(\thet\cap \latm)\big)=30$.  The red lines in Figure~\ref{figex8} indicate the gaps, \textit{i.e.}, the fibers of $\pi$ with no lattice points in $\thet$.  No matter how big of a multiple of $(\el,\ka)$ we consider, the gaps will not be closed in any scaled version of the situation.  Hence, the vertex $[\frac{2}{3},\frac{35}{3}]$ is never hit, and the associated valuation semigroup $\sg_{\flag}(\ds)$ is therefore not finitely generated.
\end{example}

\begin{figure}[h!]
\includegraphics[]{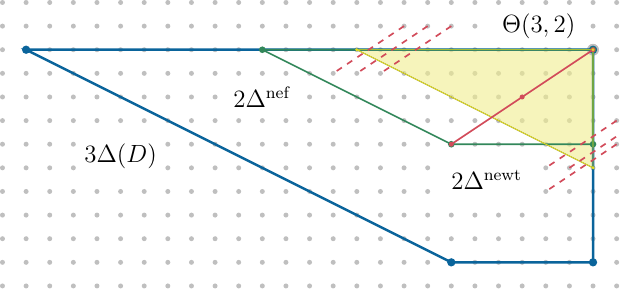}
\caption{{\bf Combinatorial view on a non-finitely generated semigroup {\boldmath $\sg_{\flag}(\ds)$}.} 
The polytopes $3\divpol{\ds}$, $2\newpolv$, $2\nefpolv$, and 
$\Theta(3,2)=(3\divpol{\ds}:2 \nefpolv)$. The dashed red lines indicate the difference between $\overline{\ee(3,2)}-1=29$ and 
$d(3,2)=35$.
}
\label{figex8}
\end{figure}

\section{Criterion for the finite generation of certain valuation semigroups} 
\label{fgSG}

We provide a criterion for the finite generation of 
strictly positive (with respect to their height functions)
semigroups in terms
of their limit polyhedra. 

\subsection{Semigroups with polyhedral limit}
\label{sgPolLim}
We start with a free abelian group $\latm$ of rank $n$, \textit{i.e.}, $\latm \cong\Z^n$, and a linear form $h\colon\latm\surj\Z$ which we call a \emph{height function}.  This induces $h_\R \colon \latm_\R= \latm\otimes_\Z\R\surj\R$, which we will often denote by $h$ as well.

Let $S\subseteq h^{-1}(\N)$ be a semigroup that is \emph{strictly positive}; \textit{i.e.}, $S\cap\ker(h)=\{0\}$.  In order to refer to the individual layers of a given height, we will write
$$
\textstyle
S_k\colonequals S\cap h^{-1}(k);\;
\mbox{\textit{i.e.}, we have }S=\bigcup_{k\in\N} S_k\;\mbox{ with } S_0=\{0\} .
$$
This setup allows us to define the enveloping cone
$$\istar_S\colonequals \ko{\cone S}\subseteq \latm_\R$$
as well as the convex limit figure
$$\poly_S\colonequals \istar_S \cap h_\R^{-1}(1)
\supseteq S_1 .$$

In the case of a valuation semigroup $\sg_{\flag}(\ds)$, the 
height 
$h \colon \Z^3 \surj \Z$ is the projection on $\el$, which then leads to the Newton--Okounkov body $\nob{\flag}{\ds}= \istar_{\sg_{\flag}(\ds)} \cap h_\R^{-1}(1) \subseteq h_\R^{-1}(1) \cong \R^2$.

\begin{definition}
\label{def-sgPolLim}
We say that $S$ has a {\em polyhedral limit} if $\poly_S$ is a 
polytope, \textit{i.e.}, if $\poly_S$ equals the convex hull of its
(finitely many) vertices.
\end{definition}

This property is fulfilled whenever the semigroup $S$ is finitely generated.  In this case, $\poly_S$ even has  rational vertices; it is a {rational polytope}.  However, the following standard example shows that the converse implication does not hold in general. 

\begin{example}
\label{ex-sgPolLim}
Let $h\colon\Z^2\to\Z, [x,y] \mapsto x+y$ be the summation map. Then
$S\colonequals \{0\}\cup(\Z_{\geq 1}\times \Z_{\geq 1})$ is not finitely generated,
but $\istar_S=\R^2_{\geq 0}$ and $\poly_S$ equals the line segment
connecting the points $[1,0]$ and $[0,1]$.
\end{example}

\subsection{Equivalent conditions for finite generation}
\label{equCondFg}

We assume that 
$S\subseteq \latm$ is a strictly positive (with respect to $h$)
semigroup that has  
polyhedral limit $\poly_S$.

\begin{definition}
\label{def-liftPoints}
We say that a point $p\in\poly_S$ {\em lifts to the semigroup $S$} (\textit{i.e.}, is a valuation point) if there exists some scalar $c \in\R_{> 0}$ with $c\cdot p\in S$.
\end{definition}

Note that in this case, both $p$ and $c$ have to be rational; \textit{i.e.}, $p\in\poly_S\cap \latm_\Q$ and $c\in\Q_{>0}$.  Hence, it is not a surprise that the assumption of the next lemma is automatically fulfilled if the semigroup $S$ is finitely generated.

\begin{lemma}
\label{lem-liftPoints}
If all vertices of $\poly_S$ lift to $S$, then they are rational
$($i.e., $\poly_S$ is a rational polytope$)$ 
and every rational point $p\in\poly_S\cap \latm_\Q$
lifts to $S$.
\end{lemma}

\begin{proof}
Let $v^1,\ldots,v^d\in\poly_S$ be linearly independent (rational) vertices
such that $p$ is contained in their convex hull. Then the unique coefficients
$\lambda_i$ in the representation
$p=\sum_{i=1}^d\lambda_i\,v^i$ have to be rational, too. Thus, we may choose an
integer $\mu$ such that $\mu \cdot\lambda_i\in\N$.
On the other hand, 
there is a joint factor $c\in\Z_{\geq 1}$ such that all multiples
$c\cdot v^i$ belong to $S$. This implies 
\begin{equation*}\pushQED{\qed}
\textstyle
\mu c\cdot p
=\sum_{i=1}^d \mu\lambda_i\cdot (c\cdot v^i)\in S .
\qedhere\popQED
\end{equation*}
\renewcommand{\qed}{}
\end{proof}

Next, we  formulate the main point of this section.

\begin{proposition}
\label{prop-liftPoints}
A semigroup $S$ with a 
polyhedral limit $\poly_S$ is finitely generated if and only if
all vertices of $\poly_S$ lift to $S$.
\end{proposition}

We have already seen that this condition is necessary for the finite
generation.
Now we will show that it is sufficient, too. 
Note that, in Example~\ref{ex-sgPolLim}, the two vertices of the line 
segment $\poly_S$ indeed do  not lift to the semigroup.

Let $S\subseteq M$ be a subsemigroup with rational polyhedral limit
$\poly(S)$ (with respect to some height function $h\colon M\to\Z$). Assume
that the vertices
and thus, by Lemma~\ref{lem-liftPoints},  
all rational points of $\poly_S$ lift to $S$.
Moreover, we may assume that 
$\istar_S$ is a full-dimensional cone.

\begin{proof}[Proof of Proposition~\ref{prop-liftPoints}]
Assume that $S$ is not finitely generated. Then $S$ has infinitely many 
indecomposable elements; \textit{i.e.},
$$\hilb \colonequals S \setminus \big( (S \setminus \{0\}) + (S \setminus \{0\}) \big)$$
is infinite; \textit{cf.}~\cite[Proposition 1.2.23]{CoxBook}.

By taking a simplicial subdivision we may, w.l.o.g., assume that the cone $\istar_S$ is simplicial and given as $\istar_S= \cone(s_1, \ldots, s_n)$.  Consider the lattice $\Lambda$ generated by $s_1, \ldots, s_n$. As $M / \Lambda$ is finite, there must be a coset $m + \Lambda$ which contains infinitely many elements of $\hilb$. Here we may choose $m$ to be a minimal representative in $\istar_S$: $m \in \istar_S \cap M$ so that $m-s_i \not\in \istar_S$ for $i = 1,\ldots, n$.

As the elements in $\istar_S \cap \hilb$ were indecomposable in $S$, they certainly are indecomposable in $\istar_S \cap S$. In particular, if we identify $(m + \Lambda) \cap \istar_S$ with $\N^n$, we obtain an infinite set of pairwise incomparable elements, in contradiction to Dickson's lemma~\cite[Chapter 4, Theorem 5]{CLO}.
\end{proof}

\section{Finite generation criterion}

\subsection{Characterising the lifting property}

The following theorem gives  a purely combinatorial criterion to check if 
our valuation semigroup   $\sg_{\flag}(\ds)$ (in the language  
of Section~\ref{setup}) is finitely generated.
Recall the definition 
$$\de(\ka/\el) = \length(\pi(\thet))$$ 
from Section~\ref{sec:shapenob} on page 
\pageref{eq:d(l,k)}.

\begin{theorem}
\label{theofg}
The point
$(1,\ka/\el,\de(\ka/\el) ) =(1,\qu,\de(\qu))$
 is a valuation point
$($i.e., a multiple of it lies in $\sg_{\flag}(\ds)$$)$ if and only if 
there exists a $\lambda \in \N$ such that 
$$\pi\colon \lambda \thetq \cap \latm \longrightarrow \pi(\lambda \thetq) \cap  \ko{\latm}$$
is surjective $($i.e.,
$  \lambda \cdot  \de(\qu) = \overline{\ee(\lambda,\lambda\qu)}
-1)$ and $\pi(\lambda\thetq)$ has endpoints in $\bM$.
 \end{theorem}
 
\begin{proof}
By definition, $(1,\qu,\de(\qu))$  is a valuation point if and only if
there exists a $\lambda \in \N$ such that we have 
$ \lambda \cdot (1,\qu,\de(\qu))  \in \sg_{\flag}(\ds) \subseteq \N^3$.
By Theorem~\ref{theoshapesg}, the latter happens exactly if 
$$0 \leq   \lambda \cdot \de(\qu) \leq \overline{\ee(\lambda,\lambda\qu)}-1,$$
where
$ \overline{\ee(\lambda,\lambda\qu)}= \# \pi(\lambda \thetq \cap \latm)$.
In addition, we see that 
 \begin{align*}
 \lambda \cdot \de(\qu)  =  \lambda \cdot \de(1,\qu)  &= \lambda \cdot \length(\pi(\thetq)) \\
 &=\length(\lambda \pi(\thetq))  =\length(\pi(\lambda  \thetq))  \\ 
 &\geq \# \pi(\lambda \thetq) \cap \ko{\latm} -1 \\
& \geq  \# \pi(\lambda\thetq \cap \latm) -1 = \overline{\ee(\lambda,\lambda\qu)}-1.
 \end{align*}
 Combining all these inequalities, we obtain the equations
  \begin{align} \label{inequ1}
 \lambda \cdot \de(\qu) = \# \pi(\lambda \thetq) \cap \ko{\latm} -1 
 \end{align}
 and 
   \begin{align} \label{inequ2}
 \# \pi(\lambda \thetq) \cap \ko{\latm}  =   \# \pi(\lambda\thetq \cap \latm) ,
 \end{align}
 where Equation~\eqref{inequ1} is equivalent to $ \pi(\lambda\thetq)$
 having end points in $\bM$, and Equation~\eqref{inequ2} to 
 $$\pi\colon \lambda \thetq \cap \latm \longrightarrow \pi(\lambda \thetq) \cap  \ko{\latm}$$
being surjective
(\textit{i.e.}, $\pi$ meets all possible lattice points in $\pi(\lambda \thetq) \cap  \ko{\latm}$).
\end{proof}

In the following, the \emph{tangent cone} of a polygon $\Theta$ at a point $r \in \Theta$ is the 
cone generated by $\Theta - r$. It is pointed if and only if $r$ is a vertex.

\begin{lemma}\label{lem_valpoint}
  Suppose that the functional $\direcn \colon M \to \Z$ takes its minimum and maximum values over $\thetq$ at the two vertices $\vminus'$ and $\vplus'$, respectively. Denote their tangent cones by $\coneminus'$ and $\coneplus'$.

  Then $(1,q,d(q))$ is a valuation point if and only if\, $1 \in \pi(\coneminus' \cap \latm)$ and $-1 \in \pi(\coneplus' \cap \latm)$.
\end{lemma}
    
\begin{proof}
    If $(1,q,d(q))$ is a valuation point, then according to Theorem~\ref{theofg}, there exists a $\lambda \in \N$ such that $\pi(\lambda \thetq)$ has vertices in $ \ko{\latm}$ and such that $\pi\colon \lambda \thetq \cap \latm \to \pi(\lambda \thetq) \cap \ko{\latm}$ is surjective.  Hence, $1 \in \pi\left(\lambda (\thetq - \vminus') \cap \latm \right) \subset \pi(\coneminus' \cap \latm)$ and similarly for $\coneplus'$.

    For the converse, assume $1 \in \pi(\coneminus' \cap \latm)$ and $-1 \in \pi(\coneplus' \cap \latm)$. We will construct a suitable scaling factor $\lambda \in \N$ for which
    $$
    \pi\colon \lambda \thetq \cap \latm \lra \pi(\lambda \thetq) \cap  \ko{\latm}
    $$
    is surjective and $\pi(\lambda \thetq)$ is a lattice polytope, in order to again apply Theorem~\ref{theofg}. 
    To this end, choose levels $\delta_{\min}\in \Q_{>0}$ and $\delta_{\max}\in \Q_{>0}$ such that 
    \begin{equation*}
      \coneminus'  \cap   \left[ \direcn \leq \delta_{\min} \right] \subset \thetq -   \vminus'  \quad  \text{and}  \quad 
        \coneplus'  \cap   \left[ \direcn \geq \delta_{\max} \right]  \subset \thetq -  \vplus'.
    \end{equation*}
 We set 
    \begin{equation*}
        \varepsilon_{\min} = \length \left( \coneminus' \cap \left[ \direcn = \delta_{\min} \right]  \right) \quad \text{and} \quad \varepsilon_{\max} = \length \left( \coneplus' \cap \left[ \direcn = \delta_{\max} \right]  \right)\,  .
    \end{equation*}
    Choose a $\lambda \in \N$ with $\lambda > \frac{1}{\varepsilon_{\min}}$ and $\lambda > \frac{1}{\varepsilon_{\max}}$  such that $\lambda \vminus'$ and $\lambda \vplus' $ are lattice points in $\latm$. Then the corresponding projection $\pi$ is surjective. To show that, we divide the image $\pi(\lambda \thetq)$ into three parts. The first $\lambda \delta_{\min}$ lattice points in $\pi (\lambda \thetq) \cap \ko\latm$ are in the image of $\pi$ because $\pi $ restricted to $\coneminus'$ is surjective by assumption. The same argument holds for the last $\lambda \delta_{\max}$ points since $\pi$ is surjective on $\coneplus'$. The points in between are hit by projecting the lattice points in $\lambda \thetq$ because all the respective fibers have length greater than $\length(\direcm)=1$, by construction. Thus $\pi$ is surjective and $(1,q,d(q))$ is a valuation point, according to Theorem~\ref{theofg}.
\end{proof}

\begin{remark}
The case where $\direcn$ takes its minimum (or maximum) not at a vertex but at an edge is actually easier to handle. As soon as $\lambda \thetq$ is a lattice polytope, the edge is an integral multiple of $\direcm$. So we can omit the first (or last) of the three parts in the above proof.
\end{remark}

\subsection{Strong decomposability}
We have seen that it is important to decide the surjectivity of the projection of lattice points in a cone in~$\latm_\R$. Next, we will translate this surjectivity into a statement in $\latn$.  To this end, we introduce the following notion.

\begin{definition}  \label{dfnsdecomp}
Let $\con \subseteq \latn_\R$ be a cone. A lattice point $\vect \in \interior(\con)\cap \latn$ is 
\emph{strongly decomposable in $\con$} if $\vect=\vect'+\vect''$ for suitable $\vect',\vect'' \in \interior(\con)\cap \latn $.  
\end{definition}

\begin{lemma} \label{lem:decomposable}
Let $\con \subseteq \latn_\R$ be a cone and 
$\vectK \in \interior(\con)\cap \latn $ a direction. 
Then the following statements are equivalent:
\begin{enumerate}
\item\label{l:d-1} We have $1 \notin \langle \hilb_{ \con^\vee}, \vectK \rangle$,
where $\hilb_{ \con^\vee}$ denotes the Hilbert basis of $ \con^\vee$. 
\item\label{l:d-2} The direction $\vectK$ is strongly decomposable in $\con$. 
\item\label{l:d-3} The closure of the one-parameter subgroup
$\lambda^{\vectK}(\C^*)\subseteq\tor$
in $\toric(\con)$ is singular. 
\end{enumerate}
\end{lemma}

\begin{proof}
\eqref{l:d-1} and \eqref{l:d-3} are equivalent: The $1$-parameter subgroup
represented by $\vectK$ can always be extended to 
$\lambda^{\vectK}\colon\C\to\toric(\con)$. On the dual level of regular
functions, however, this corresponds to
$$\langle\,\cdot\, ,\vectK \rangle\colon\C[\con^\vee\cap \latm]\longrightarrow\C[\N].$$ 
The latter map is 
surjective if and only if $1\in \langle \con^\vee\cap M,\vectK\rangle$.

\eqref{l:d-1} $\Rightarrow$ \eqref{l:d-2}: By assumption, there exist primitive lattice points 
$s^0,s^1\in M \setminus \con^\vee$ such that the line segment between 
them lying on $[\vectK=1]$
contains no interior lattice point but intersects $\con^\vee$. Note that 
$\{s^0,s^1\}$ is a $\Z$-basis because
the lattice triangle $\conv(0,s^0,s^1)$ is unimodular. Then 
$$\interior\left(\cone\left(s^0,s^1\right)\right) \supset \con^\vee\quad \text{and thus}\quad \cone\left(t^0,t^1\right) \subset \interior(\con),$$
where $\{t^0,t^1\}$ denotes the basis dual to $\{s^0,s^1\}$.  By the definition of the dual basis, this yields $\langle s^i, \vectK \rangle = 1 = \langle s^i, t^0+t^1 \rangle$ for all $i$; \textit{i.e.}, the two linear functionals coincide on the basis. Therefore, $\vectK = t^0+t^1$.

\eqref{l:d-2} $\Rightarrow$ \eqref{l:d-1}: Let $\vectK$ be strongly decomposable in $\con$.  Then $\vectK=\vect'+\vect''$ for some $\vect',\vect'' \in \interior(\con)\cap \latn$.  Therefore, we have $\langle \op , \vectK \rangle = \langle \op, \vect' \rangle +\langle \op, \vect'' \rangle \in \Z$ for $\op \in \hilb_{ \con^\vee}$.  Since $\vect'$ and $\vect''$ lie in the interior of $\con$, both summands are positive.  Thus, $1 \notin \langle \hilb_{ \con^\vee} , \vectK \rangle$.
\end{proof}

\begin{figure}[h!]
     \centering
     \begin{subfigure}[b]{0.45\textwidth}
         \centering
               \includegraphics[]{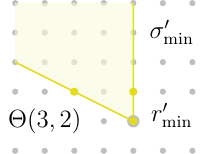}
         \caption{}
              \end{subfigure}
          \hfill
     \begin{subfigure}[b]{0.45\textwidth}
         \centering
        \includegraphics[]{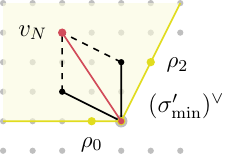}
         \caption{}
     \end{subfigure}
       \caption{{\bf Strongly decomposable primitive element.} 
       (A) The polytope $\Theta(3,2)$ with  tangent cone $\coneminus'$ 
       at the $\direcn$-extremal vertex $\vminus'$. 
       (B) The dual cone $(\coneminus')^\vee$ together with the strong 
       decomposition $(-2,3)=(-2,1)+(0,2)$ of $\direcn$ inside it.}
\label{fig:decomposable}
\end{figure}

\begin{example}\label{ex9}
We continue Example~\ref{ex8} and  fix $(\el,\ka)= (3,2)$. 
Then $$\thet = \conv([0,0],[-10,0],[0,-5])$$ 
(\textit{cf.} Figure~\ref{figex8}).
Its two $\direcn$-extremal vertices are $\vplus'=[-10,0]$ and $\vminus'=[0,-5]$.
Thus,
the corresponding tangent cones $\coneminus'$ and $\coneplus'$ 
are given as $\cone([0,1],[-2,1])$ and $\cone([1,0],[2,-1])$,
respectively (\textit{cf.} Section~\ref{altthet}). 
Hence the direction $\direcn=(-2,3)$ is contained in the interior of 
$(\coneminus')^{\vee}=\cone((-1,0),(1,2))$. 
It is strongly decomposable in $(\coneminus')^{\vee}$ because $\direcn=(-2,1)+(0,2)$
(\textit{cf.}~Figure~\ref{fig:decomposable}). 
An application of Lemma~\ref{lem:decomposable} yields that we do not obtain $1$ via 
$\langle \hilb_{ \coneminus'}, \direcn \rangle$.
\end{example}

Now we are ready to formulate our finite generation criterion.
The essence is that it suffices to check strong decomposability of $\direcn$ in two specific cones which we now define.

\begin{definition} \label{dfndline}
Given $\divpol{\ds}$ and $\direcn$, we define the rational line segment $\dline \subseteq \divpol{\ds}$ as a line segment $\divpol{\ds} \cap [\direcn = c]$ of maximal length orthogonal to $\direcn$.  We call its vertices $\dlineone$ and $\dlinetwo \in \latm_\R$ and its length $\quhat$ so that $\dlinetwo-\dlineone = \quhat \direcm$.  Moreover, we denote by $\ebigone$ the part of the edge of $\divpol{\ds}$ with vertex $\dlineone$ lying in the half-plane $[\direcn \geq c]$, and by $\ebigtwo$ the part of the edge of $\divpol{\ds}$ with vertex $\dlinetwo$ also lying  in the half-plane $[\direcn \geq c]$.  The cone $\consmall \subseteq \latn_\R$ is the cone generated by the inner normal vectors of $\ebigone$ and $\ebigtwo$.
  
 In the same manner, we define the line segments $\esmallone$ and $\esmalltwo$ contained in $[\direcn \leq c]$, which yield the cone~$\conbig$.
\end{definition}

Observe that $\thetq \neq \emptyset$ (and thus $\de(\qu) \ge 0)$ if and only if $q \in [0,\quhat]$.
         
\begin{theorem}\label{thmfgsg}
The valuation semigroup $\sg_{\flag}(\ds)$ is finitely generated if and only if $\direcn$ is not strongly decomposable in $\conbig$ and $-\direcn$ is not strongly decomposable in $\consmall$.
\end{theorem}

\begin{proof}
Combining Lemmas~\ref{lem_valpoint} and~\ref{lem:decomposable}, we see that $\sg_{\flag}(\ds)$ is finitely generated if and only if for every rational $\qu \in [0,\quhat]$, the vector $\direcn$ is not strongly decomposable in $(\coneminus')^\vee$ and $-\direcn$ is not strongly decomposable in $(\coneplus')^\vee$.

Now we use that for all $q \in [0,\quhat)$, 
 $$ (\conbig)^{\vee} \subseteq \coneminus' \quad \text{and} \quad (\consmall)^{\vee} \subseteq \coneplus',$$ so $\direcn$ is strongly decomposable in $(\coneminus')^\vee$ for some $q$ if and only if it is strongly decomposable in $\conbig$, and correspondingly for $-\direcn$.
\end{proof}

\begin{corollary}
The valuation semigroup $\sg_{\flag}(\ds)$ is finitely generated if and only if the morphism $\PP^1 \to \var'$ given by $\direcn$ is a smooth embedding, where $X'$ is the toric variety associated with the fan generated by $\conbig$ and $\consmall$.
\end{corollary}

\begin{example}\label{ex7gon}
We apply Theorem~\ref{thmfgsg} to the $7$-gon $\divpol{\ds}$ (\textit{cf.} Figure~\ref{fig7gon}\subref{fig7gon1}) with vertices $[4,1]$, $ [7,2]$, $[9,3]$, $[6,5]$, $[1,8]$, $[1,7]$, and $[2,4]$ from \cite[Example 4.8]{castravet2020blown}, which is a \emph{good polytope} in the language of \textit{op.~cit.}.

\begin{figure}[h!]
\centering
     \begin{subfigure}[h]{0.45\textwidth}
         \centering
        \includegraphics[]{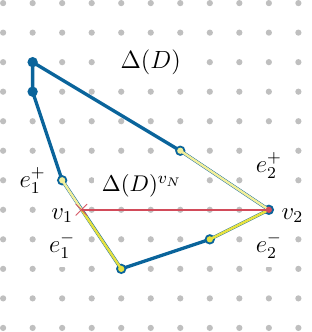}
             \caption{}
             \label{fig7gon1}
         \end{subfigure}
     \begin{subfigure}[h]{0.45\textwidth}
  \centering
        \includegraphics[]{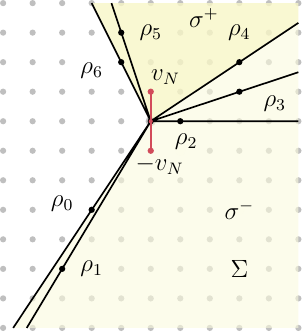}
        \caption{}
             \label{fig7gon2}
     \end{subfigure}
       \caption{{\bf The {good} {\boldmath$7$}-gon {\boldmath$\divpol{\ds}$}.} 
       (A) The polytope $\divpol{\ds}$ having seven vertices together with 
       the rational line segment $\dline = \divpol{\ds} \cap [\direcn=3]$ and
       its two vertices $\dlineone =[8/3,3]$, $\dlinetwo = [9,3]$.
       (B) The normal fan $\fan$ of $\divpol{\ds}$ having seven rays $\rho_i$ ($0 \leq i \leq 6$)
       together with the two cones $\conbig= \cone(\rho_4,\rho_6) \ni \direcn$  and 
       $\consmall= \cone(\rho_0,\rho_4) \ni - \direcn$.}
\label{fig7gon}
\end{figure} 

In \textit{op.~cit.},  the authors construct examples of projective toric surfaces whose blow-ups at the general point have a non-polyhedral pseudo-effective cone. In particular, this is the case for projective toric surfaces associated with good polytopes; \textit{cf.}~\cite[Definition 4.3, Theorem 4.4]{castravet2020blown}.

Consider the projective toric surface $\var = \toric(\fan)$ associated with 
the normal fan $\fan$ of $\divpol{\ds}$  with rays 
\begin{align*}
&\rho_0= (-2,-3),\;
\rho_1=(-3,-5),\; 
\rho_2=(1,0),\; 
\rho_3=(3,1), \\
&\rho_4=(3,2),\; 
\rho_5=(-1,3), 
\text{ and } \rho_6=(-1,2)
\end{align*}
(\textit{cf.} Figure~\ref{fig7gon}\subref{fig7gon2}) and an admissible flag $\flag \colon \var \supseteq \curv \supseteq \{1\}$ on $\var$ with $\direcn=(0,1)$.

To apply Theorem~\ref{thmfgsg}, we compute the following data: 
$$\dline = \divpol{\ds} \cap [\direcn=3] = \conv([8/3,3],[9,3]) ,$$ \textit{i.e.}, $\dlineone =[8/3,3]$ and $\dlinetwo = [9,3]$.  The inner normal vectors of $\ebigone$, $\ebigtwo$ and $\esmallone$, $\esmalltwo$ are $\rho_4$, $\rho_0$ and $\rho_4$, $\rho_6$, respectively.  Therefore,
$$\conbig= \cone(\rho_4,\rho_6) \ni (0,1)= \direcn\quad \text{and}\quad \consmall= \cone(\rho_4,\rho_0) \ni (0,-1)=-\direcn .$$ Thus the associated semigroup $\sg_{\flag}(\ds)$ is finitely generated because $\direcn= (0,1)$ is not strongly decomposable in $\conbig$ and $- \direcn= (0,-1)$ is not strongly decomposable in $\consmall$.
\end{example}

\subsection{Varying $\boldsymbol{\direcn}$ and $\boldsymbol{\ds}$}

The strategy to obtain a finitely generated semigroup by choosing an appropriate direction $\direcn$ does not always work. For the following example, there is no direction $\direcn$ that works.

\begin{example} \label{final:ex}
Consider the ample divisor $\ds$ associated with the polytope $\divpol{\ds}$ depicted in Figure~\ref{finalex}\subref{finalex1} on the toric variety $\var = \toric(\fan)$ corresponding to the fan $\fan$ depicted in Figure~\ref{finalex}\subref{finalex2}.  We claim that no matter what $\direcn \in N$ we pick, the resulting semigroup $\sg_{\flag}(\ds)$ will not be finitely generated.

\begin{figure}[h!]
\centering
     \begin{subfigure}[h]{0.45\textwidth}
         \centering
        \includegraphics[]{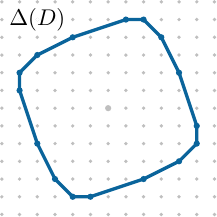}
             \caption{}
             \label{finalex1}
         \end{subfigure}
     \begin{subfigure}[h]{0.45\textwidth}
  \centering
        \includegraphics[]{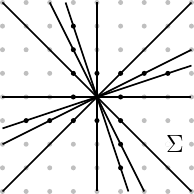}
              \caption{}
             \label{finalex2}
     \end{subfigure}
       \caption{{\bf {\boldmath$\sg_{\flag}(\ds)$} non-finitely generated for all {\boldmath$\direcn$}.} 
       (A) The polytope $\divpol{\ds}$ associated with an ample divisor $\ds$ on $\var = \toric(\fan)$.
       (B) The fan $\fan$ of $\toric(\fan)$ with $16$ rays.}
\label{finalex}
\end{figure} 

We will use our characterization in Theorem~\ref{thmfgsg}.
As $\divpol{\ds}$ is centrally symmetric, the longest line segment
$\dline$ in Definition~\ref{dfndline} will pass through the origin,
whatever $\direcn$.

We distinguish two cases: either the endpoints of the segment $\dline$ are vertices of $\divpol{\ds}$, or they belong to the interior of an edge. Up to symmetry, there are four vertices and four edges to consider. We will carry out the argument for one vertex and for one edge. The others are left to the reader.

If $\dline$ hits the interior of the edges $\edge_1^{\pm}, \edge_2^{\pm}$ indicated in Figure~\ref{finalexI}\subref{finalex1}, then $\direcn$ must belong to the interior of the red region in Figure~\ref{finalexI}\subref{finalex2}.  In this case, the cones $\sigma^\pm$ from Definition~\ref{dfndline} will be the two half-planes bounded by the dotted line which is perpendicular to the direction of $\edge_1^{\pm}, \edge_2^{\pm}$. But all lattice vectors in the red region are strongly decomposable in their half-plane as they all have lattice distance greater than one from the dotted line. (The vectors which are not strongly decomposable are the vectors at distance one; \textit{i.e.}, they lie on the dashed lines.)

If, on the other hand, $\dline$ contains the vertices $\dlineone, \dlinetwo$ indicated in Figure~\ref{finalexI}\subref{finalex5}, then $\direcn$ is determined up to sign, as are $\sigma^\pm$ (\textit{cf.} Figure~\ref{finalexI}\subref{finalex6}).  Again, we see that $\direcn$ is strongly decomposable.  (And again, the vectors which are not strongly decomposable are the vectors which lie on the dashed lines.)
\end{example}

\begin{figure}[t!]
\centering
     \begin{subfigure}[h]{0.45\textwidth}
         \centering
        \includegraphics[]{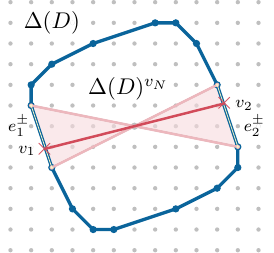}
             \caption{}
             \label{finalex3}
                 \vspace{0.3cm}
         \end{subfigure}
     \begin{subfigure}[h]{0.45\textwidth}
  \centering
        \includegraphics[]{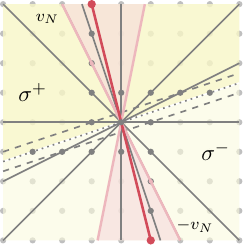}
        \caption{}
             \label{finalex4}
     \end{subfigure}
     \begin{subfigure}[h]{0.45\textwidth}
         \centering
        \includegraphics[]{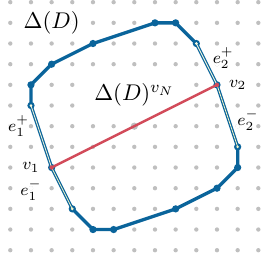}
                     \caption{}
             \label{finalex5}
         \end{subfigure}
     \begin{subfigure}[h]{0.45\textwidth}
  \centering
        \includegraphics[]{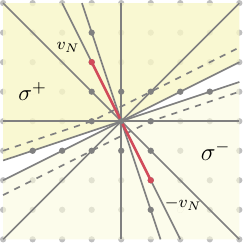}
                \caption{}
             \label{finalex6}
     \end{subfigure}
     \caption{{\bf Illustration of Example~\ref{final:ex}.}
     (A) The line segment $\dline$ containing the interior of the edges $\edge_1^{\pm}, \edge_2^{\pm}$.
     (B)  The possible region for $\pm \direcn$  in red  together with the two cones $\sigma^\pm$.
       (C) The line segment $\dline$ hitting two vertices $\dlineone, \dlinetwo$ of $\divpol{\ds}$.
       (D)  The cones $\sigma^\pm$ containing $\pm \direcn$.
       }
\label{finalexI}
\end{figure} 

\begin{proposition}\label{thm:finitelygenerated}
  Given a fan $\Sigma$ and a direction $\direcn$, the valuation
  semigroup $\sg_{\flag}(\ds)$ is finitely generated for all ample
  divisors $\ds$ on $\var$ if and only if $\direcn$ and $-\direcn$
  are not strongly decomposable in $\con$ for all cones 
  $\con \subseteq \latn_\R$ generated by rays of the given fan $\fan$.
\end{proposition}

\begin{proof}
``$\Leftarrow$'': Theorem~\ref{thmfgsg}.
``$\Rightarrow$'':
Assume there exists a cone $\con$ built from rays of $\fan$
such that $\direcn$ is strongly decomposable in $\con$. 
We will construct an ample divisor $\ds$ 
such that  $\sg_{\flag}(\ds)$ is not finitely generated. 

In a first step, we build a torus-invariant divisor $\ds_{\thet}=\sum_{\ray \in \fan(1)}{b_\ray \ds_\ray}$ whose associated polytope $\thet$ has a vertex $\vminus'$ with tangent cone $\coneminus'=\con^\vee$.  For that we choose coefficients $b_\ray \in \Z$ such that $\ds_{\thet}$ has positive intersection numbers with all curves corresponding to rays $ \ray \nsubseteq \interior(\con)$. Then we relax the inequalities at $\vminus'$ by choosing the coefficients $b_\ray \gg 0$ for all $\ray \subseteq \interior(\con)$.  This guarantees that $\con^\vee$ is the tangent cone of $\thet$ at its vertex $\vminus'$.

Now, we want to define an ample torus-invariant divisor $\ds$ such that $\thet=(\divpol{\ds}: \nefpolv)$.  As an intermediate step, set $\ds'=\sum_{\ray \in \fan(1)}{a'_\ray \ds_\ray} \colonequals \curvbar+ \ds_{\thet}$.  Then by construction, the associated polytope $\divpol{\ds'}$ contains the Minkowski sum $\thet+ \nefpolv$. Moreover, the defining inequalities of $\thet+ \nefpolv$ coincide with the ones for the polytope $\divpol{\ds'}$ for all rays $\ray \nsubseteq \interior(\con)$.

In general, $\ds'$ is not nef and in particular not ample because it might have negative intersection with curves associated with remaining rays $\ray \subseteq \interior(\con)$.  Hence, as a last step, we define a torus-invariant divisor $\ds=\sum_{\ray \in \fan(1)}{a_\ray\ds_\ray}$ with $a_\ray=a'_\ray$ for all $\ray \in \fan(1)\setminus \interior(\con)$.  For the remaining rays $\ray \subseteq \interior(\con)$, choose coefficients $a_\ray$ small enough such that $\ds$ is ample and big enough such that $\thet+ \nefpolv \subseteq \divpol{\ds}$.  Then we have $\thet=(\divpol{\ds}: \nefpolv)$, and $\thet$ has a vertex $\vminus'$ with tangent cone $\coneminus'=\con^\vee$.  Since $\direcn$ is strongly decomposable in $\con$, it follows from Theorem~\ref{thmfgsg} that the semigroup $\sg_{\flag}(\ds)$ is not finitely generated.
\end{proof}

\begin{figure}[h!]
     \centering
     \begin{subfigure}[h]{0.25\textwidth}
           \includegraphics[]{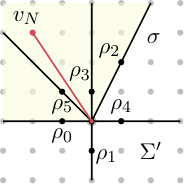}
         \caption{}
          \label{figex101}
     \end{subfigure}
     \begin{subfigure}[h]{0.25\textwidth}
         \centering
        \includegraphics[]{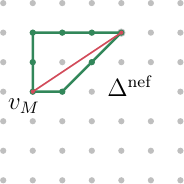}
         \caption{}
          \label{figex102}
     \end{subfigure}
     \begin{subfigure}[h]{0.25\textwidth}
         \centering
        \includegraphics[]{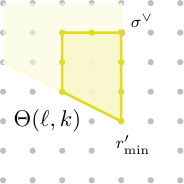}
         \caption{}
           \label{figex103}
     \end{subfigure}
     
     \vspace{0.3cm}
       \begin{subfigure}[h]{0.45\textwidth}
         \centering
        \includegraphics[]{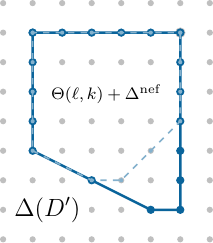}
         \caption{}
                 \label{figex104}
     \end{subfigure}
     \hspace{-0.5cm}
        \begin{subfigure}[h]{0.45\textwidth}
         \centering
        \includegraphics[]{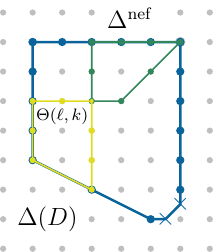}
         \caption{}
                 \label{figex105}
     \end{subfigure}
     \caption{{\bf Illustration of Proposition~\ref{thm:finitelygenerated}.}
     (A) The set of  rays $\fan'(1)=\fan(1) \cup \{  \rho_4,\rho_5 \}$, 
     a primitive element $\direcn=(-2,3) \in \latn$, and the cone $\sigma$ in which $\direcn$ is
     strongly decomposable.
     (B)  The Newton polytope $\newpolv$ given by $\direcm=[-3,-2]$ and the polytope $\nefpolv$ corresponding to 
       $\curvbar$ in $\var'=\TV(\fan')$.
       (C) The polytope $\thet$ having a vertex $\vminus'$ with tangent cone $\sigma^\vee$. 
       (D) The polytope $\divpol{\ds'}$ corresponding to the non-ample divisor $\ds'$ containing $\thet+\nefpolv$.
       (E) The polytope $\divpol{\ds}$ corresponding to the ample divisor $\ds$ with $\thet=(\divpol{\ds}: \nefpolv)$. 
}
  \label{figex10}
\end{figure}

\begin{example}\label{ex10}
We illustrate the construction from the proof of Proposition~\ref{thm:finitelygenerated} for a modification of our running Example~\ref{ex9} with $\direcn = (-2,3) \in \latn$.  Let $\var'=\TV(\fan')$ be the toric surface associated with the fan $\fan'$, where $\fan'(1)=\fan(1) \cup \{ \rho_4 , \rho_5 \}$ and $\rho_4= (1,0)$, $\rho_5=(-1,1)$.  Then (as in Example~\ref{ex9}) $\direcn = (-2,1)+(0,2)$ is strongly decomposable in $\sigma=\cone(\rho_0,\rho_2)$ (\textit{cf.}~Figure~\ref{figex10}\subref{figex101}).  Moreover,
$$\curvbar=7\ds_2 + 2\ds_3+3\ds_4$$
in $\var'=\TV(\fan')$
with 
$\nefpolv= \conv([0,0],[-3,0],[-3,-2],[-2,-2])$
 (\textit{cf.}~Figure~\ref{figex10}\subref{figex102}).
Choose the divisor 
 $\ds_{\thet}$ as $$\ds_{\thet}=6\ds_2+4\ds_3+2\ds_4+6\ds_5,$$
 \textit{i.e.},
 $\thet= \conv([0,0],[-2,0],[-2,-2],[0,-3])$
 (\textit{cf.} Figure~\ref{figex10}\subref{figex103}).
We obtain the divisor
$$\ds'=\ds_{\thet}+\curvbar=13\ds_2+6\ds_3+5\ds_4+6\ds_5$$
with $\divpol{\ds'} = \conv([0,0],[-5,0],[-5,-4],[-1,-6],[0,-6])$ and $\thet + \nefpolv \subseteq \divpol{\ds'}$ (\textit{cf.} Figure~\ref{figex10}\subref{figex104}).  Note that ${\ds'}$ is not ample since we have  $\ds' \cdot \ds_5=0$.

To make it ample, adjust the coefficient of $\ds_5$ to $5.5$. Then $\ds$ is ample with $\divpol{\ds} = \conv([0,0],[-5,0],$ $[-5,-4],[-1,-6],[-0.5,-6],[0,-5.5])$ (\textit{cf.}~Figure~\ref{figex10}\subref{figex105}).  Moreover, we have $\thet=(\divpol{\ds}: \nefpolv)$.
\end{example}


\end{document}